\documentclass[twoside]{siamltex}

\usepackage{amssymb,amsmath,exscale}
\usepackage{color}
\usepackage{xspace}
\usepackage[draft,ulem=normalem]{changes}

\newtheorem{remark}[theorem]{{\itshape Remark}}

\newcommand{\Pb}{\mbox{\rm (P)}\xspace}
\newcommand{\uad}{\mathbf{U_{\rm ad}}}
\renewcommand{\div}{\operatorname{div}}
\newcommand{\bY}{\mathbf{Y}}
\newcommand{\by}{\mathbf{y}}
\newcommand{\bu}{\mathbf{u}}
\newcommand{\bv}{\mathbf{v}}
\newcommand{\bz}{\mathbf{z}}
\newcommand{\bef}{\mathbf{f}}
\newcommand{\bg}{\mathbf{g}}
\newcommand{\be}{\mathbf{e}}

\newcommand{\p}{\mathfrak{p}}
\newcommand{\q}{\mathfrak{q}}

\newcommand{\bphi}{{\boldsymbol\phi}}
\newcommand{\bpsi}{{\boldsymbol\psi}}
\newcommand{\bvarphi}{{\boldsymbol\varphi}}

\newcommand{\bLd}{{\mathbf{L}^2(\Omega)}}
\newcommand{\bLqd}{{\mathbf{L}^2(Q)}}

\newcommand{\bLf}{{\mathbf{L}^4(\Omega)}}

\newcommand{\bLs}{{\mathbf{L}^{s}(\Omega)}}

\newcommand{\bWoT}{{\mathbf{W}(0,T)}}
\newcommand{\bWqpoT}{{\mathbf{W}_{q,p}(0,T)}}

\newcommand{\bWrsoT}{{\mathbf{W}_{r,s}(0,T)}}
\newcommand{\bV}{\mathbf{V}}

\newcommand{\bBrs}{{\mathbf{B}_{s,r}(\Omega)}}
\newcommand{\bBqp}{{\mathbf{B}_{p,q}(\Omega)}}

\newcommand{\bM}{{\mathbf{M}(\omega)}}
\newcommand{\bWs}{{\mathbf{W}_s(\Omega)}}
\newcommand{\bWsc}{{\mathbf{W}_{s'}(\Omega)}}
\newcommand{\bWp}{{\mathbf{W}_p(\Omega)}}
\newcommand{\bWpc}{{\mathbf{W}_{\!p'}(\Omega)}}
\newcommand{\bWop}{{\mathbf{W}_0^{1,p}(\Omega)}}
\newcommand{\bWos}{{\mathbf{W}_0^{1,s}(\Omega)}}
\newcommand{\bWopc}{{\mathbf{W}_0^{1,p'}(\Omega)}}

\newcommand{\bWmop}{{\mathbf{W}^{-1,p}(\Omega)}}

\newcommand{\bH}{{\mathbf{H}}}
\newcommand{\bX}{{\mathbf{X}}}
\newcommand{\bHo}{{\mathbf{H}_0^1(\Omega)}}
\newcommand{\bHmo}{{\mathbf{H}^{-1}(\Omega)}}
\newcommand{\bHt}{{\mathbf{H}^2(\Omega)}}
\newcommand{\bHto}{{\mathbf{H}^{2,1}(Q)}}

\newcommand{\bCo}{{\mathbf{C}_0(\omega)}}
\newcommand{\lco}{{L^1(I;\bCo)}}

\newcommand{\lmoq}{{L^q(I;\bM)}}
\newcommand{\lmo}{{L^\infty(I;\bM)}}
\newcommand{\mo}{{M(\omega)}}

\newcommand{\bna}{\mathbf{\nabla}}
\newcommand{\mF}{\mathcal{F}}
\newcommand{\mA}{\mathcal A}

\newcommand{\mL}{\mathcal{L}}

\newcommand{\mY}{\mathcal{Y}}

\newcommand{\proj}{\mbox{\rm proj}}

\begin{document}

\title{Optimal Control of the 2D Evolutionary Navier-Stokes Equations with Measure Valued Controls\thanks{The first author was supported by Spanish Ministerio de Econom\'{\i}a, Industria y Competitividad under research project MTM2017-83185-P. The second was supported by the ERC advanced grant 668998 (OCLOC) under the EU’s H2020 research program.}}

\author{Eduardo Casas\thanks{Departamento de Matem\'{a}tica Aplicada y Ciencias de la Computaci\'{o}n, E.T.S.I. Industriales y de Telecomunicaci\'on, Universidad de Cantabria, 39005 Santander, Spain (eduardo.casas@unican.es).}
    \and Karl Kunisch\thanks{Institute for Mathematics and Scientific Computing, University of Graz, Heinrichstrasse 36, A-8010 Graz, Austria (karl.kunisch@uni-graz.at).}}


\maketitle

\begin{abstract}
In this paper, we consider an optimal control problem for the two-dimensional evolutionary Navier-Stokes system. Looking for sparsity, we take controls as functions of time taking values in a space of Borel measures. The cost functional does not involve directly the control but we assume some constraints on them. We prove the well-posedness of the control problem and derive necessary and sufficient conditions for local optimality  of the controls.
\end{abstract}

\begin{AMS}
35Q30, 
49J20, 
49J52, 
49K20, 
49K40  
\end{AMS}

\begin{keywords}
Navier-Stokes equations, Borel measures, Sparsity, First and second order optimality conditions
\end{keywords}

\pagestyle{myheadings} \thispagestyle{plain} \markboth{E.~CASAS AND K.~KUNISCH}{Control of Navier-Stokes Equations with Measures}

\section{Introduction}
\label{S1}

In this paper we investigate the following optimal control problem
\[
\Pb \quad \min_{\bu \in \uad}J(\bu) = \frac{1}{2}\int_Q|\by_\bu(x,t) - \by_d(x,t)|^2\, dx\, dt,
\]
where $\uad = \{\bu \in L^\infty(0,T;\bM) : \|\bu(t)\|_\bM \le \gamma \text{ for a.a. } t \in (0,T)\}$ with $0 < \gamma < \infty$, and $\by$ and $\bu$ are related by the Navier-Stokes system
\begin{equation}
\left\{\begin{array}{l}\displaystyle\frac{\partial \by}{\partial t} -\nu\Delta\by + (\by \cdot \bna)\by + \nabla\p = \bef_0 + \chi_\omega\bu\ \text{ in } Q = \Omega \times I,\\[1.2ex]\div\by = 0 \ \text{ in } Q, \ \by = 0 \ \text{ on } \Sigma = \Gamma \times I,\ \by(0) = \by_0 \text{ in } \Omega.\end{array}\right.
\label{E1.1}
\end{equation}
Here, $I = (0,T)$ with $0 < T < \infty$, $\Omega$ denotes a bounded domain in $\mathbb{R}^2$ with a $C^3$ boundary $\Gamma$, and $\omega$ is a relatively closed subset of $\Omega$. We denote $\bM = M(\omega) \times M(\omega)$, where $M(\omega)$ is the space of real and regular Borel measures in $\omega$. In the cost functional $J$, the target $\by_d \in \bLqd$ is fixed.  Regarding the state equation, $\nu > 0$ is the kinematic viscosity coefficient, $\chi_\omega\bu$ denotes the extension of $\bu$ by zero outside $\omega$, and $\bef_0$ is a given element of $L^q(I,\bWmop)$ with $\bWmop = W^{-1,p}(\Omega) \times W^{-1,p}(\Omega)$, where
\begin{equation}
\frac{4}{3} \le p < 2\quad \text{and}\quad  q > \frac{2p}{p - 1}
\label{E1.2}
\end{equation}
are fixed. Observe that the previous assumptions imply that $q > 4$. For the initial condition we can take $\by_0 \in \bWop$ such that $\div{\by_0} = 0$. A more general choice for $\by_0$ will be given later.

Our motivation for the analysis of measure-valued controls is two-fold. On the one hand there it is the genuine interest in low-order regularity of the controls, on the other hand it relates to their sparsity promoting structure. Indeed, it has been observed and analyzed in much previous work that the optimal controls are typically zero over subsets of the domain, whereas they  would simply be 'small', but not zero, if they would be replaced by a control in a Hilbert space, for example. We refer, exemplarily to the work in \cite{Casas2017, CCK2012,  KT2016}, which treats these phenomena for equations of diffusion type as well as for wave equations. In these papers the sparsity promoting terms is part of the cost, whereas in \cite{Casas-Kunisch2019B}
the measure valued term appears as a constraint like in $\uad$ above.
 It should also be mentioned that in case  the measure-valued setting is replaced by an $L^1$ formulation together with  $L^2$ constraints or penalties, again sparsity phenomena occur, but the optimal controls are, of course, functions in this case rather than measures \cite{CHW2017, HSW2015}. 

In the literature,  the optimal control of the Navier-Stokes equations has received much attention, we refer exemplarily to \cite {Abergel-Temam90, BTZ00, DG08, DI94, HK01, TW06}, and the monograph \cite{G12} and the survey \cite{CC16}. The controls are always considered as functions in these contributions. Apparently the only work on measure valued optimal controls in the case of the Navier Stokes equations is \cite{Casas-Kunisch2019} which treats the stationary case.

For evolutionary  Navier Stokes equations with forcing functions of low regularity, allowing for measure-valued forcing, very little analysis has been carried out even for the state equation by itself. We are only aware of   \cite{Serre1983}, where the right hand side in \eqref{E1.1} is chosen in $W^{1,\infty}(I; \bWmop)$, with $\bWmop=\bigotimes_{i=1}^d W^{-1,p}(\Omega)$,  $d\in \{2,3\}$,  and $p\in (\frac{d}{2}, 2]$. It is mentioned there, that likely the result is not optimal. In our previous work \cite{Casas-Kunisch2020} we have obtained the necessary well-posedness results for \eqref{E1.1} which are required for the study of optimal control problems. Thus the current work is the first one which considers optimal control for evolutionary Navier Stokes equations with measure-valued controls.

When formulating optimal control problems some restrictions on the class of admissible controls  are essential to guarantee existence of minimizers, to be obtained by the standard method of the calculus of variations. Such restrictions are also  well motivated by applications. One possible choice consists in adding a properly chosen control cost to the cost-functional $J$ in $\Pb$. In our case it could be a  term of the form  $ \frac{\beta}{q} \int^T_0\|\bu(t)\|^q_\bM \,dt$, where $\beta$ is a positive weight.  For technical reasons $q=2$ seems not to be possible, since it does not imply sufficient temporal regularity on the class of admissible controls. From the analytical point of view it would suffice to take $q>4$. But we prefer to rather work with pointwise constraints in time. In this way we arrive at the class $\uad$ and the problem formulation  chosen in $\Pb$. This choice of temporal pointwise  constraints, also poses new challenges in  deriving both necessary and sufficient  second order optimality conditions, regardless of the measure-valued norm in space.

Let us comment further  on the norm in $\bM$ appearing in \Pb. First, we recall that $M(\omega)$ is a Banach space when endowed with the norm
\[
\|u\|_{M(\omega)} = \sup_{\|\phi\|_{C_0(\omega)} \le 1}\int_\omega\phi(x)\, du(x) = |u|(\omega),
\]
where $C_0(\omega) = \{\phi \in C(\bar\omega) : \phi(x) = 0 \ \forall x \in \partial\omega \cap \Gamma\}$ is a separable Banach space, and $|u|$ represents the total variation measure of $u$; see \cite[page 130]{Rudin70}. Note that $C_0(\omega) \neq C(\bar\omega)$ only in the case that $\bar \omega$ has a nonempty intersection with $\Gamma$.

For vector-valued measures we define
\begin{equation}
\|\bu\|_\bM = \max(\|u_1\|_{M(\omega)},\|u_2\|_{M(\omega)}),
\label{E1.3}
\end{equation}
which makes $\bM$ a Banach space. It is the dual space of $\bCo = C_0(\omega) \times C_0(\omega)$ when it is endowed with the norm $\|\bphi\|_\bCo = \|\phi_1\|_{C_0(\omega)} + \|\phi_2\|_{C_0(\omega)}$.

Hereafter we denote by $\lmo$ the space of weakly measurable functions $\bu:(0,T) \longrightarrow \bM$ satisfying $\|\bu\|_\lmo = \text{ess\,sup}_{t \in I}\|\bu(t)\|_\bM < \infty$. This norm makes $\lmo$ a Banach space and guarantees that it can be identified with the dual of $\lco$, where the duality relation is given by
\[
\langle \bu,\bz\rangle_{\lmo,\lco} = \int_0^T\langle \bu(t), \bz(t) \rangle_{\bM,\bCo}\, dt.
\]
The reader is referred to \cite[section 8.14.1 and Proposition 8.15.3]{Edwards1965} for the different notions of measurability and \cite[Theorem 8.18.2]{Edwards1965} for the duality identification.
(The distinction between weak and strong measurability is not required for the space $\lco$ because $\bCo$ is separable and hence both notions are equivalent; see \cite[Theorem 8.15.2]{Edwards1965}.). Observe that $\lmo$ is a subspace of $L^\infty(I;\bWmop)$ for every $p< 2$. Indeed, the embedding $\bWopc \subset \mathbf{C}_0(\Omega) \subset \bCo$ implies that the duality $\langle \bu(t),\bz\rangle$ is well defined for every $\bu \in \lmo$ and $\bz \in \bWopc$, and we have
\begin{align*}
|\langle \bu(t),\bz\rangle_{\bM,\bCo}| &\le \|\bu(t)\|_\bM\|\bz\|_\bCo\\
&\le C_{p,\Omega}\|\bu(t)\|_\bM\|\bz\|_\bWopc \le C_{p,\Omega}\|\bu\|_\lmo\|\bz\|_\bWopc
\end{align*}
for a.a. $t \in I$ and a constant $C_{p,\Omega}$ depending only on $p$ and $\Omega$. Analogously, we have that $\lmoq$ is a Banach space for the norm
\[
\|\bu|_\lmoq = \Big(\int_0^T\|\bu\|^q_\bM\, dt\Big)^{1/q},
\]
dual of $L^{q'}(I;\bCo)$. Obviously, the embedding $\lmo \subset \lmoq$ holds. The right hand side of the state equation, $\bef_0 + \chi_\omega\bu$, is well defined as an element of $L^q(I;\bWmop)$ for every $\bu \in \lmoq$.\vspace{2mm}

\noindent{\it{Structure of paper}.} In the following section, well-posed results on the state equation relevant for the remainder of the paper are summarized. Here we can rely on results from \cite{Casas-Kunisch2020}. Existence of solutions to $\Pb$ and first order optimality conditions are the contents of section 3. Necessary and sufficient second order optimality conditions will be given in section 4. This requires further  detailed analysis of the state equations and its linearization in functions spaces of low regularity.
\bigskip

\noindent\textbf{NOTATION}

In this paper, we denote $\bWos = W_0^{1,s}(\Omega) \times W_0^{1,s}(\Omega)$ for $s \in (1,\infty)$, and we choose as the norm in $\bWos$
\[
\|\by\|_\bWos = \|\nabla\by\|_\bLs = \left(\int_\Omega|\nabla\by|^s\, dx\right)^{\frac{1}{s}}  = \left(\int_\Omega[|\nabla y_1|^2 + |\nabla y_2|^2]^{\frac{s}{2}}\, dx\right)^{\frac{1}{s}}.
\]
We also consider the spaces
\begin{align*}
&\bH = \text{closure of }\{\bphi\in \mathbf{C}^\infty_0(\Omega):  \text{ div}\,  \bphi =0 \}\text{ in }\bLd,\\
&\bWs = \{\by \in \bWos : \div\by = 0\}.
\end{align*}
For $s = 2$ we set $\bHo = \mathbf{W}_0^{1,2}(\Omega)$ and $\bV = \mathbf{W}_2(\Omega)$.

We also define the following spaces
\begin{align*}
&\bWoT = \{\by \in L^2(I;\bV) : \frac{\partial \by}{\partial t} \in L^2(I;\bV')\},\\
&\bWrsoT = \{\by \in L^r(I;\bWs) : \frac{\partial \by}{\partial t} \in L^r(I;\bWsc')\},\\
&\bV^{2,1}(0,T) = \{\by \in L^2(I;\bHt \cap \bV) : \frac{\partial\by}{\partial t} \in L^2(I;\bH)\}
\end{align*}
with $r, s \in (1,\infty)$, endowed with the norms
\begin{align*}
&\|\by\|_\bWoT = \|\by\|_{L^2(I;\bHo)} + \|\frac{\partial \by}{\partial t}\|_{L^2(I;\bV')},\\
&\|\by\|_\bWrsoT = \|\by\|_{L^r(I;\bWos)} + \|\frac{\partial \by}{\partial t}\|_{L^r(I;\bWsc')},\\
&\|\by\|_{\bV^{2,1}(0,T)} = \|\by\|_{L^2(I;\bHt)} + \|\frac{\partial \by}{\partial t}\|_{L^2(I;\bH)}.
\end{align*}
Obviously these are reflexive Banach spaces, and $\bWoT = \bWrsoT$ if $r = s = 2$. Moreover, $\bWoT$ and $\bV^{2,1}(0,T)$ are Hilbert spaces.

Now we consider the interpolation space $\bBrs = (\bWsc',\bWs)_{1-1/r,r}$. From \cite[Chap. III/4.10.2]{Amann1995} we know that $\bWrsoT \subset C([0,T];\bBrs)$ and the trace mapping $\by \in \bWrsoT \to \by(0) \in \bBrs$ is surjective. If $r = s = 2$, then it is known that $\mathbf{B}_{2,2}(\Omega) = (\bV',\bV)_{\frac{1}{2},2} = \bH$. Hence, the embedding $\bWoT \subset C([0,T];\bH)$ holds; see \cite[Page 22, Proposition I-2.1]{Lions-Magenes68} and \cite[Page 143, Remark 3]{Triebel1978}.

\section{Analysis of the state equation}
\label{S2}
\setcounter{equation}{0}

The aim of this section is to study the well-posedness and differentiability of the mapping control-to-state. The results presented in this section are based on the analysis carried out in \cite{Casas-Kunisch2020}.

Let us consider the Banach space $\bY_0 = \bH + \bBqp$ with the norm
\[
\|\by_0\|_{\bY_0} =  \inf_{\by = \by_1 + \by_2}\|\by_1\|_\bLd + \|\by_2\|_\bBqp.
\]
It will be assumed that the initial state $\by_0$ in \eqref{E1.1} is an element of $\bY_0$. Now we introduce the following spaces:
\begin{align*}
&\bY = [L^2(I;\bV) \cap L^\infty(I;\bH)] + L^q(I;\bWp),\\
&\mY = \bWoT + \bWqpoT.
\end{align*}
They are Banach spaces with the norms
\begin{align*}
&\|\by\|_Y = \inf_{\by = \by_1 + \by_2}\|\by_1\|_{L^2(I;\bHo)} + \|\by_1\|_{L^\infty(I;\bLd)} + \|\by_2\|_{L^q(I;\bWop)},\\
&\|\by\|_\mY = \inf_{\by = \by_1 + \by_2}\|\by_1\|_\bWoT + \|\by_2\|_\bWqpoT.
\end{align*}
Note that $\mY \subset \bY$. Moreover, since $\bWoT$ and $\bWqpoT$ are reflexive spaces, then $\mY$ is reflexive as well. The solution of \eqref{E1.1} will be found in $\mY$.

\begin{definition}
Given $\bef_0 \in L^q(I,\bWmop)$, $\bu \in \lmoq$ and $\by_0 \in \bY$, we say that $\by \in \mY$ is a solution of \eqref{E1.1} if
\begin{equation}
\left\{
\begin{array}{l}
\displaystyle\frac{d}{dt}\langle\by(t),\bpsi\rangle_{\bWpc)' ,\bWpc} + a(\by(t),\bpsi) + b(\by(t),\by(t),\bpsi) \\
= \langle\bef_0(t),\bpsi\rangle_{\bWmop,\bWopc} + \langle\bu(t),\bpsi\rangle_{\bM,\bCo}\ \text{ in } (0,T),\ \ \forall \bpsi \in \bWpc,\\
\by(0) = \by_0,\end{array}\right.
\label{E2.1}
\end{equation}
where the system of differential equations is satisfied in the distribution sense and
\begin{align*}
&a(\by(t),\bpsi) = \nu\int_\Omega\nabla\by(x,t) : \nabla\bpsi(x)\, dx = \nu\sum_{i = 1}^2\int_\Omega\nabla y_i(x,t)\nabla\psi_i(x)\, dx,\\
&b(\by(t),\by(t),\bpsi) =  \int_\Omega[\by(t) \cdot \nabla]\by(t)\cdot\nabla\bpsi\, dx.
\end{align*}
A distribution $\p$ in $Q$ is called an associated pressure if the equation
\[
\frac{\partial \by}{\partial t} -\nu\Delta\by + (\by \cdot \bna)\by + \nabla\p = \bef_0 + \chi_\omega\bu\ \text{ in } Q
\]
is satisfied in the distribution sense. Then, $(\by,\p)$ is called a solution of \eqref{E1.1}.
\label{D2.1}
\end{definition}

Given $\by$ satisfying \eqref{E2.1}, the pressure $\p$ is obtained by using De Rham's theorem; see \cite[Lemma IV-1.4.1]{Sohr2001}. As pointed out in Section 1, the embeddings $\bWoT \subset C([0,T];\bH)$ and $\bWqpoT \subset C([0,T];\bBqp)$ hold. Hence, $\mY \subset C([0,T];\bY_0)$ and, consequently, the initial condition $\by(0) = \by_0$ with $\by_0 \in \bY_0$ makes sense.

The next theorem establishes the well-posedness of the state equation \eqref{E1.1}. It is an immediate consequence of \cite[Theorem 2.2]{Casas-Kunisch2020}.

\begin{theorem}
Suppose that $(\bef_0,\by_0) \in L^q(I,\bWmop) \times \bY_0$ and that \eqref{E1.2} holds. Then, system \eqref{E2.1} has a unique solution $(\by,\p) \in \mY \times W^{-1,q}(I;L^p(\Omega)/\mathbb{R})$ for every $\bu \in \lmoq$. Furthermore, there exists a nondecreasing function $\eta_{p,q}:[0,\infty) \longrightarrow [0,\infty)$ with $\eta_{p,q}(0) = 0$ such that
\begin{equation}
\|\by\|_\mY \le \eta_{p,q}\Big(\|\bef_0\|_{L^q(I;\bWpc')} + \|\bu\|_{L^q(I;\bWmop)} + \|\by_0\|_{\bY_0}\Big).
\label{E2.2}
\end{equation}
\label{T2.1}
\end{theorem}

Now, we introduce the mapping $G:\lmoq \longrightarrow \mY$ associating to each control $\bu \in \lmoq$ the solution $\by_\bu \in \mY$ of \eqref{E1.1}. Then we have the following differentiability result.

\begin{theorem}
$G$ is of class $C^\infty$. Further, given $\bu, \bv, \bv_1, \bv_2 \in \lmoq$ we have that $\bz_\bv = G'(\bu)\bv$ and $\bz_{\bv_1,\bv_2} = G''(\bu)(\bv_1,\bv_2)$ are the unique solutions in $\mY$ of the Oseen systems
\begin{equation}
\left\{\begin{array}{l}\displaystyle\frac{\partial \bz}{\partial t} -\nu\Delta\bz + (\by_\bu \cdot \bna)\bz + (\bz \cdot \bna)\by_\bu + \nabla\q = \chi_\omega\bv\ \text{ in } Q,\\[1.2ex]\div\bz = 0 \ \text{ in } Q, \ \bz = 0 \ \text{ on } \Sigma,\ \bz(0) = 0 \text{ in } \Omega,\end{array}\right.
\label{E2.3}
\end{equation}
and
\begin{equation}
\left\{\begin{array}{l}\displaystyle\frac{\partial \bz}{\partial t} -\nu\Delta\bz + (\by_\bu \cdot \bna)\bz + (\bz \cdot \bna)\by_\bu + \nabla\q = - (\bz_{\bv_2} \cdot \bna)\bz_{\bv_1} - (\bz_{\bv_1} \cdot \bna)\bz_{\bv_2}\ \text{ in } Q,\\[1.2ex]\div\bz = 0 \ \text{ in } Q, \ \bz = 0 \ \text{ on } \Sigma,\ \bz(0) = 0 \text{ in } \Omega,\end{array}\right.
\label{E2.4}
\end{equation}
respectively, where $\by_\bu = G(\bu)$ and $\bz_{\bv_i} = G'(\bu)\bv_i$ for $ i = 1, 2$.
\label{T2.2}
\end{theorem}

\begin{proof}
Let $G_0:L^q(I;\bWmop) \longrightarrow \mY$ be defined by $G_0(\bef) = \by_\bef$ with $\by_\bef$ the solution of the system
\begin{equation}
\left\{\begin{array}{l}\displaystyle\frac{\partial \by}{\partial t} -\nu\Delta\by + (\by \cdot \bna)\by + \nabla\p = \bef\ \text{ in } Q = \Omega \times I,\\[1.2ex]\div\by = 0 \ \text{ in } Q, \ \by = 0 \ \text{ on } \Sigma = \Gamma \times I,\ \by(0) = \by_0 \text{ in } \Omega.\end{array}\right.
\label{E2.5}
\end{equation}
Then, we have that $G(\bu) = (G_0 \circ B)(\bu)$ with $B:\lmoq \longrightarrow L^q(I;\bWmop)$ given by $B\bu = \bef_0 + \chi_\omega\bu$. The statement of the theorem is a straightforward consequence of the chain rule and \cite[Theorerm 5.1]{Casas-Kunisch2020}.
\end{proof}

We finish this section proving the a continuity result for $G$.

\begin{theorem}
Let $\{\bu_k\}_{k = 1}^\infty \subset \lmoq$ be a sequence such that $\bu_k \stackrel{*}{\rightharpoonup}\bu$ in $\lmoq$, then $\by_{\bu_k} \rightharpoonup \by_\bu$ in $\mY$ and $\by_{\bu_k} \to \by_\bu$ in $L^2(I;\bH_{2p})$, where $\bH_{2p} = \bH \cap \mathbf{L}^{2p}(\Omega)$.
\label{T2.3}
\end{theorem}

\begin{proof}
The boundedness of $\{\bu_k\}_{k = 1}^\infty$ in $\lmoq$ along with the estimate \eqref{E2.2} implies the boundedness of $\{\by_{\bu_k}\}_{k = 1}^\infty$ in $\mY$. Since $\mY$ is reflexive, there exists a subsequence, denoted in the same way, such that $\by_{\bu_k} \rightharpoonup \by$ in $\mY$. Now, we pass to the limit in equation \eqref{E2.1} satisfied by every pair $(\by_{\bu_k},\bu_k)$. In this process, the only difficulty is found in the nonlinear term $b(\by_{\bu_k},\by_{\bu_k},\bpsi)$. To deal with it we use a compact embedding. Using the Sobolev embeddings $\bV \subset \bH_{2p} \subset \bV^*$ and $\bWp \subset \bH_{2p} \subset \bWpc^*$, which are compact, we have the compactness of the embeddings $\bWoT \subset L^2(I;\bH_{2p})$ and $\bWqpoT \subset L^q(I;\bH_{2p})$, see \cite[Theorem III-2.1]{Temam79}. Since $q > 4$, we get that the embedding $\mY \subset L^2(I;\bH_{2p})$ is compact. Hence, we deduce that $\by_{\bu_k} \to \by$ strongly in $L^2(I;\bH_{2p})$. Finally, given $\bpsi \in \bWpc$ and using the antisymmetric property of $b$ we get
\[
b(\by_{\bu_k},\by_{\bu_k},\bpsi) = -b(\by_{\bu_k},\bpsi,\by_{\bu_k}) \to -b(\by,\bpsi,\by) = b(\by,\by,\bpsi)\ \text{ strongly in }\ L^1(I).
\]
Therefore, $\by$ satisfies equation \eqref{E2.1} and, hence, $\by = \by_{\bu}$. Since every convergent subsequence of $\{\by_{\bu_k}\}_{k = 1}^\infty$ converges to the same limit $\by_{\bu}$, we conclude that the whole sequence converges as claimed in the theorem to $\by_{\bu}$.
\end{proof}

\section{Existence of solutions of \Pb and first order optimality conditions}
\label{S3}
\setcounter{equation}{0}

We start this section by proving the existence of solutions for the control problem \Pb. Then, we show the differentiability of the cost functional and deduce the first order necessary optimality conditions. From these conditions we infer the sparsity properties of the stationary controls.

\begin{theorem}
There exists at least one solution $\bar\bu$ of \Pb.
\label{T3.1}
\end{theorem}

\begin{proof}
First, we observe that $\uad$ is the closed ball of $\lmo$ centered at $\mathbf{0}$ and radius $\gamma$. Moreover, $\lco$ is a separable Banach space and $\lmo = \lco^*$. Hence, given a minimizing sequence $\{\bu_k\}_{k = 1}^\infty$ for \Pb, there exists a subsequence, denoted in the same way, such that $\bu_k \stackrel{*}{\rightharpoonup} \bar\bu$ in $\lmo$. Then, Theorem \ref{T2.3} implies that $\by_{\bu_k} \to \by_{\bar\bu}$ in $L^2(Q)$. Therefore, $J(\bu_k) \to J(\bar\bu) = \inf\Pb$ holds. Thus, $\bar\bu$ is a solution of \Pb.
\end{proof}

Before stating the optimality conditions satisfied by a solution of \Pb, we analyze the differentiability of the cost functional.
\begin{theorem}
The cost functional $J:\lmoq \longrightarrow \mathbb{R}$ is of class $C^\infty$ and the following identities hold
\begin{align}
&J'(\bu)\bv = \int_0^T\langle\bv(t),\bvarphi_\bu(t)\rangle_{\bM,\bCo}\, dt,\label{E3.1}\\
&J''(\bu)\bv^2 = \int_Q\left\{|\bz_\bv|^2 + 2(\bz_\bv\cdot\bna)\bvarphi_\bu\bz_\bv\right\}\, dx\, dt, \label{E3.2}
\end{align}
for all $\bv \in \lmoq$, where $\bz_\bv = G'(\bu)\bv$ and $\bvarphi_\bu \in \bV^{2,1}(0,T)$ is the adjoint state, the unique solution along with the pressure $\pi_\bu$ of
\begin{equation}
\left\{\begin{array}{l}\displaystyle-\frac{\partial\bvarphi}{\partial t} - \nu\Delta\bvarphi - (\by_\bu \cdot \bna)\bvarphi - (\bna\bvarphi)^T\by_\bu + \nabla \pi = \by_\bu - \by_d\ \text{ in } Q,\\[1.2ex]\div\bvarphi = 0 \ \text{ in } Q, \ \bvarphi = 0 \ \text{ on } \Sigma,\ \bvarphi(T) = 0 \ \text{ in } \Omega.\end{array}\right.
\label{E3.3}
\end{equation}
\label{T3.2}
\end{theorem}

\begin{proof}
The differentiability of $J$ is a consequence of the chain rule and Theorem \ref{T2.2}. The expressions \eqref{E3.1} and \eqref{E3.2} follow from \eqref{E2.3}, \eqref{E2.4} and \eqref{E3.3}. We only have to prove that \eqref{E3.3} has a unique solution that belongs to $\bV^{2,1}(0,T)$. To this end, let us consider the classical operator associated with the Stokes system $A:\bV \longrightarrow \bV'$ given by $\langle A\bpsi,\bphi\rangle_{\bV',\bV} = a(\bpsi,\bphi)$ $\forall \bpsi, \bphi \in \bV$. As usual, we take a basis $\{\bpsi_j\}_{j = 1}^\infty$ of $\bV$ formed by eigenfunctions of $A$: $A\bpsi_j = \lambda_j\bpsi_j$ with $\{\lambda_j\}_{j = 1}^\infty \subset (0,\infty)$, $j \ge 1$. We assume that $\{\bpsi_j\}_{j = 1}^\infty$ is orthonormal for the Hilbert product in $\bH$: $(\psi_i,\bpsi_j)_\bLd = \delta_{ij}$. Let us denote by $\bV_k$ the subspace generated by $\{\bpsi_1,\ldots,\bpsi_k\}$. Following the classical Faedo-Galerkin approach, we discretize \eqref{E3.3}
\begin{equation}
\left\{\hspace{-0.25cm}
\begin{array}{l}
\displaystyle -\frac{d}{dt}(\bvarphi_k(t),\bpsi_j)_\bLd + a(\bvarphi_k(t),\bpsi_j) - b(\by_{\bu}(t),\bvarphi_k(t),\bpsi_j) \vspace{2mm}\\
- b(\bpsi_j,\bvarphi_k(t),\by_{\bu}(t)) = (\by_{\bu}(t) - \by_d(t),\bpsi_j)_\bLd\text{ in } (0,T),\ 1 \le j \le k,\\
\bvarphi_k(T) = 0,
\end{array}\right.
\label{E3.4}
\end{equation}
where $\bvarphi_k(t) = \sum_{j = 1}^kg_{k,j}(t)\bpsi_j$. Arguing analogously as in \cite[Proof of Theorem 2.7]{Casas-Kunisch2020}, we infer the existence and uniqueness of a solution $\bvarphi_k$ satisfying the estimate
\begin{equation}
\|\bvarphi_k\|_\bWoT \le \eta_0(\|\by_{\bu}\|_\bY)\|\by_{\bu} - \by_d\|_\bLqd \ \ \forall k,
\label{E3.5}
\end{equation}
where $\eta_0:[0,\infty) \longrightarrow [0,\infty)$ is a nondecreasing function vanishing at 0. Moreover, as in \cite{Casas-Kunisch2020}, we can prove that $\{\bvarphi_k\}_{k = 1}^\infty$ converges weakly in $\bWoT$ to the unique solution $\bvarphi_{\bu}$ of \eqref{E3.3}. Moreover, $\bvarphi_{\bu}$ also satisfies the estimate \eqref{E3.5}. It remains to prove the $\bV^{2,1}(0,T)$ regularity. To this end, we split the proof into two parts.\vspace{0.25cm}

{\em I - Estimate of $\|\bvarphi_k\|_{L^2(I;\bHt\cap\bV)}$.} First, we observe that
\[
A\bvarphi_k = \sum_{j = 1}^kg_{k,j}A\bpsi_j = \sum_{j = 1}^k\lambda_jg_{k,j}\bpsi_j.
\]
Multiplying equation \eqref{E3.4} by $\lambda_jg_{k,j}(t)$ and taking the sum from $j = 1$ to $k$ we infer
\begin{align*}
&-\frac{d}{dt}(\bvarphi_k(t),A\bvarphi_k(t))_\bLd + a(\bvarphi_k(t),A\bvarphi_k(t)) - b(\by_{\bu}(t),\bvarphi_k(t),A\bvarphi_k(t))\\
& - b(A\bvarphi_k(t),\bvarphi_k(t),\by_{\bu}(t)) = (\by_{\bu}(t) - \by_d(t),A\bvarphi_k(t))_\bLd.
\end{align*}
Using the identities established in \cite[Page 372]{Boyer-Fabrie2013}, the above identity yields
\begin{align}
&-\frac{1}{2}\frac{d}{dt}\|\bvarphi_k(t)\|^2_\bHo + \|A\bvarphi_k(t))\|^2_\bLd = (\by_{\bu}(t) - \by_d(t),A\bvarphi_k(t))_\bLd\notag\\
& +  b(\by_{\bu}(t),\bvarphi_k(t),A\bvarphi_k(t)) + b(A\bvarphi_k(t),\bvarphi_k(t),\by_{\bu}(t)).\label{E3.6}
\end{align}
Now, we estimate the right hand side of this identity. First we get
\begin{align*}
|b(\by_{\bu}(t),\bvarphi_k(t),&A\bvarphi_k(t))| \le \|\by_{\bu}(t)\|_\bLf\|\nabla\bvarphi_k(t)\|_\bLf\|A\bvarphi_k(t))\|_\bLd\\
&\le C_1\|\by_{\bu}(t)\|_\bLf\|\nabla\bvarphi_k(t)\|^{1/2}_\bLd\|\bvarphi_k(t)\|^{1/2}_\bHt\|A\bvarphi_k(t))\|_\bLd\\
&\le C_2\|\by_{\bu}(t)\|_\bLf\|\bvarphi_k(t)\|^{1/2}_\bHo\|A\bvarphi_k(t)\|^{3/2}_\bLd\\
&\le \frac{C_2^4}{4}\Big(\frac{9}{2}\Big)^3\|\by_{\bu}(t)\|^4_\bLf\|\bvarphi_k(t)\|^2_\bHo + \frac{1}{6}\|A\bvarphi_k(t)\|^2_\bLd.
\end{align*}
Above we have used a Gagliardo inequality (see \cite[Proposition III.2.35]{Boyer-Fabrie2013}), the $\bHt$ estimates for the solution of the Stokes problem $\|\by\|_\bHt \le C\|A\by\|_\bLd$ \cite[Theorem IV.5.8]{Boyer-Fabrie2013}, and Young's inequality.

The estimate for $b(A\bvarphi_k(t),\bvarphi_k(t),\by_{\bu}(t))$ is exactly the same. Therefore, inserting these estimates in \eqref{E3.6} and using again Young's inequality we get
\begin{align*}
&-\frac{1}{2}\frac{d}{dt}\|\bvarphi_k(t)\|^2_\bHo + \|A\bvarphi_k(t))\|^2_\bLd \le \|\by_{\bu}(t) - \by_d\|_\bLd\|A\bvarphi_k(t))\|_\bLd\notag\\
&+ C_3\|\by_{\bu}(t)\|^4_\bLf\|\bvarphi_k(t)\|^2_\bHo + \frac{1}{3}\|A\bvarphi_k(t)\|^2_\bLd\notag\\
& \le \frac{3}{2}\|\by_{\bu}(t) - \by_d\|^2_\bLd + C_3\|\by_{\bu}(t)\|^4_\bLf\|\bvarphi_k(t)\|^2_\bHo + \frac{1}{2}\|A\bvarphi_k(t)\|^2_\bLd,
\end{align*}
which implies
\begin{align}
&-\frac{d}{dt}\|\bvarphi_k(t)\|^2_\bHo + \|A\bvarphi_k(t))\|^2_\bLd\notag\\
&\le 3\|\by_{\bu}(t) - \by_d\|^2_\bLd + 2C_3\|\by_{\bu}(t)\|^4_\bLf\|\bvarphi_k(t)\|^2_\bHo. \label{E3.7}
\end{align}
Let us prove that $\by_{\bu} \in L^4(I;\bLf)$. Since $\by_{\bu} \in \mY$, we can write it in the form $\by_{\bu} = \by_1 + \by_2$ with $\by_1 \in \bWoT$ and $\by_2 \in \bWqpoT$. Using again a Gagliardo inequality we obtain
\[
\|\by_1(t)\|^4_\bLf \le C_4\|\by_1(t)\|^2_\bLd\|\by_1(t)\|^2_\bHo \le C_4\|\by_1\|^2_{L^\infty(I;\bLd)}\|\by_1(t)\|^2_\bHo.
\]
The embeddings $\bWoT \subset L^2(I;\bHo)$ and $\bWoT \subset L^\infty(I;\bLd)$ and the above inequality imply $\by_1 \in L^4(I;\bLf)$. On the other hand, since $\bWqpoT \subset L^q(I;\bWp) \subset L^4(I;\bLf)$,  recall \eqref{E1.2}, we infer that $\by_2 \in L^4(I;\bLf)$. Then, $\by_{\bu} \in L^4(I;\bLf)$ holds. Now, integrating \eqref{E3.7} in $[t,T]$ and using that $\bvarphi_k(T) = 0$ it follows
\[
\|\bvarphi_k(t)\|^2_\bHo \le 3\|\by_{\bu} - \by_d\|^2_\bLqd + 2C_3\int_t^T\|\by_{\bu}(s)\|^4_\bLf\|\bvarphi_k(s)\|^2_\bHo \, ds\ \forall t \in I.
\]
Applying Gronwall inequality we infer
\begin{equation}
\|\bvarphi_k\|_{L^\infty(I;\bHo)} \le \sqrt{3}\|\by_{\bu} - \by_d\|_\bLqd\exp{\Big(C_3\|\by_{\bu}\|^4_{L^4(I;\bLf)}\Big)}\ \ \forall k \ge 1.
\label{E3.8}
\end{equation}
Finally, integrating \eqref{E3.7} in $[0,T]$ and inserting \eqref{E3.8} we obtain
\begin{align*}
&\|A\bvarphi_k\|_\bLqd\\
&\le \sqrt{3}\|\by_\bu - \by_d\|_\bLqd\left\{1 + \sqrt{2C_3}\|\by_\bu\|^2_{L^4(I;\bLf)}\exp{\Big(C_3\|\by_\bu\|^4_{L^4(I;\bLf)}\Big)}\right\}.
\end{align*}
Once again, with $\|\by\|_\bHt \le C\|A\by\|_\bLd$ \cite[Theorem IV.5.8]{Boyer-Fabrie2013} we   deduce from the above estimate $\forall k \ge 1$
\begin{align}
&\|\bvarphi_k\|_{L^2(I;\bHt)}\notag\\
& \le C\|\by_\bu - \by_d\|_\bLqd\left\{1 + \hat{C}\|\by_\bu\|^2_{L^4(I;\bLf)}\exp{\Big(\hat{C}^2\|\by_\bu\|^4_{L^4(I;\bLf)}\Big)}\right\}.
\label{E3.9}
\end{align}

{\em II - Estimate of $\|\bvarphi'_k\|_{L^2(I;\bH)}$.} Multiplying equation \eqref{E3.4} by $-g'_{k,j}\bpsi_j$, adding the resulting identities from $j = 1$ to $k$, using the orthogonality of $\{\bpsi_j\}_{j = 1}^\infty$ in $\bH$ and integrating in $[0,T]$ we get
\begin{align*}
&\|\bvarphi'_k\|^2_\bLqd - \frac{1}{2}\int_0^T\frac{d}{dt}a(\bvarphi_k(t),\bvarphi_k(t))\, dt = -\int_0^T(\by_\bu(t) - \by_d(t),\bvarphi'_k(t))_\bLd\, dt\\
&-\int_0^Tb(\by_{\bu}(t),\bvarphi_k(t),\bvarphi'_k(t))\, dt - \int_0^Tb(\bvarphi'_k(t),\bvarphi_k(t),\by_{\bu}(t))\, dt.
\end{align*}
Now, taking into account that $\bvarphi_k(T) = 0$  it follows from the above identity
\begin{align}
&\|\bvarphi'_k\|^2_\bLqd \le \|\by_\bu - \by_d\|_\bLqd\|\bvarphi_k'\|_\bLqd\notag\\
&+ \Big|\int_0^Tb(\by_{\bu}(t),\bvarphi_k(t),\bvarphi'_k(t))\, dt + \int_0^Tb(\bvarphi'_k(t),\bvarphi_k(t),\by_{\bu}(t))\, dt\Big|.\label{E3.10}
\end{align}
With the Gagliardo and Young inequalities we obtain
\begin{align*}
|b(\by_\bu(t),\bvarphi_k(t),&\bvarphi'_k(t))| \le \|\by_\bu\|_\bLf\|\nabla\bvarphi_k\|_\bLf\|\bvarphi'_k\|_\bLd\\
&\le C_1\|\by_{\bu}(t)\|_\bLf\|\nabla\bvarphi_k(t)\|^{1/2}_\bLd\|\bvarphi_k(t)\|^{1/2}_\bHt\|\bvarphi'_k(t))\|_\bLd\\
&\le \frac{3C_1^2}{2}\|\by_{\bu}(t)\|^2_\bLf\|\nabla\bvarphi_k(t)\|_\bLd\|\bvarphi_k(t)\|_\bHt + \frac{1}{6}\|\bvarphi'_k(t))\|^2_\bLd.
\end{align*}
The same estimate is valid for $|b(\bvarphi'_k(t),\bvarphi_k(t),\by_{\bu}(t))|$. Inserting these estimates in \eqref{E3.10} and using Schwarz's inequality we find
\begin{align*}
&\left|\int_0^Tb(\by_\bu(t),\bvarphi_k(t),\bvarphi'_k(t))\, dt + \int_0^Tb(\bvarphi'_k(t),\bvarphi_k(t),\by_\bu(t))\, dt\right| \\
&\le 3C_1^2\|\bvarphi_k\|_{L^\infty(I;\bHo)}\|\by_\bu\|_{L^4(I;\bLf)}^2\|\bvarphi_k\|_{L^2(I;\bHt)} + \frac{1}{3}\|\bvarphi'_k(t))\|^2_\bLd.
\end{align*}
This estimate, \eqref{E3.10} and Young's inequality lead to
\begin{align*}
&\|\bvarphi'_k\|^2_\bLqd \le \frac{3}{2}\|\by_\bu - \by_d\|^2_\bLqd + \frac{1}{6}\|\bvarphi'_k\|^2_\bLqd\\
&+ 3C_1^2\|\bvarphi_k\|_{L^\infty(I;\bHo)}\|\by_\bu\|_{L^4(I;\bLf)}^2\|\bvarphi_k\|_{L^2(I;\bHt)} + \frac{1}{3}\|\bvarphi'_k(t))\|^2_\bLd,
\end{align*}
whence
\begin{align}
\|\bvarphi'_k\|_\bLqd &\le \sqrt{3}\|\by_\bu - \by_d\|_\bLqd\notag\\
&+ \sqrt{6}C_1\|\by_\bu\|_{L^4(I;\bLf)}\|\bvarphi_k\|^{1/2}_{L^\infty(I;\bHo)}\|\bvarphi_k\|^{1/2}_{L^2;\bHt)}. \label{E3.11}
\end{align}
From \eqref{E3.8} and \eqref{E3.9} the boundedness of $\{\bvarphi'_k\}_{k = 1}^\infty$ in $\bLqd$ follows. Therefore, $\bvarphi'_\bu \in \bLqd$ holds and with the first part of the proof we conclude that $\bvarphi_\bu \in \bV^{2,1}(0,T)$.
\end{proof}

Let us note that the estimates \eqref{E2.2}, \eqref{E3.5}, \eqref{E3.8}, \eqref{E3.9} and \eqref{E3.11} yields
\begin{equation}
\|\bvarphi_\bu\|_{\bV^{2,1}(0,T)} \le \eta\Big(\|\bef_0\|_{L^q(I;\bWpc')} + \|\bu\|_\lmoq + \|\by_0\|_{\bY_0}\Big)\|\by_\bu - \by_d\|_\bLqd.
\label{E3.12}
\end{equation}
for some non-decreasing monotone function $\eta:[0,\infty) \longrightarrow [0,\infty)$.

Next, we prove the first order necessary optimality conditions. Since \Pb is not a convex problem, it is convenient to discuss necessary optimality conditions in the context of local solutions. Here, we say that $\bar\bu$ is a local solution of \Pb if there exists a neighborhood $\mA$ of $\bar\bu$ in $\lmo$ such that $J(\bar\bu) \le J(\bu)$ for all $\bu \in \mA$. If the inequality is strict for all $\bu \in \mA$ with $\bu \neq \bar\bu$, we say that $\bar\bu$ is a strict local solution. We will also consider local solutions in the $L^q(I;\bWmop)$ topology. Let us observe that the continuous embedding $\lmo \subset L^q(I;\bWmop)$ implies that any local solution in the $L^q(I;\bWmop)$ topology is also a local solution in the $\lmo$ topology.

\begin{theorem}
Let us assume that $\bar\bu$ is a local solution of \Pb with associated state $\bar\by$. Then, there exists a unique element $\bar\bvarphi \in \bV^{2,1}(0,T)$ satisfying
\begin{align}
&\left\{\begin{array}{l}\displaystyle-\frac{\partial\bar\bvarphi}{\partial t} - \nu\Delta\bar\bvarphi - (\bar\by \cdot \bna)\bar\bvarphi - (\bna\bar\bvarphi)^T\bar\by + \nabla \bar\pi = \bar\by - \by_d\ \text{ in } Q,\\[1.2ex]\div\bar\bvarphi = 0 \ \text{ in } Q, \ \bar\bvarphi = 0 \ \text{ on } \Sigma,\ \bar\bvarphi(T) = 0 \ \text{ in } \Omega,\end{array}\right.\label{E3.13}\\
&\left\{\begin{array}{l}\text{if } \bar\varphi_i(t) \not\equiv 0, \text{ then } \|\bar u_i(t)\|_{M(\omega)} = \gamma \text{ and }\\\textrm{Supp}(\bar u_i^+(t)) \subset \{x \in \omega : \bar\varphi_i(x,t) = -\|\bar\varphi_i(t)\|_{C_0(\omega)}\},\\
\textrm{Supp}(\bar u_i^-(t)) \subset \{x \in \omega : \bar\varphi_i(x,t) = +\|\bar\varphi_i(t)\|_{C_0(\omega)}\},\end{array}\right.\label{E3.14}
\end{align}
for $i = 1, 2$ and almost every point $t \in I$, where $\bar u_i(t) = \bar u^+_i(t) - \bar u^-_i(t)$ is the Jordan decomposition of the measure $\bar u_i(t)$.
\label{T3.3}
\end{theorem}

\begin{proof}
From Theorem \ref{T3.2} we know the existence and uniqueness of $\bar\bvarphi \in \bV^{2,1}(0,T)$ satisfying \eqref{E3.13}. From the expression for $J'$ given in \eqref{E3.1} and using the convexity of $\uad$ we have
\[
0 \le J'(\bar\bu)(\bu - \bar\bu) = \int_0^T\langle\bu(t) - \bar\bu(t),\bar\bvarphi(t)\rangle_{\bM,\bCo}\, dt\quad \forall \bu \in \uad.
\]
This is equivalent to
\begin{equation}
\int_0^T\langle u(t),\bar\varphi_i(t)\rangle_{\mo,C_0(\omega)}\, dt \le  - \int_0^T\langle \bar u_i(t),\bar\varphi_i(t)\rangle_{\mo,C_0(\omega)}\, dt, \ \ i = 1, 2.
\label{E3.15}
\end{equation}
for every $u$ satisfying $\|u\|_{L^\infty(I;\mo)} \le \gamma$.

Since $\bar\varphi_i:\bar\Omega \times I\to \mathbb{R}$ is a Caratheodory function (continuous with respect to the first variable and measurable with respect to the second), there exists a measurable selection $t \in I \mapsto x_t \in \bar\Omega$ such that $|\bar\varphi_i(x_t,t)| = \|\bar\varphi_i(t)\|_{C_0(\omega)}$; see \cite[Chapter 8, Theorem 1.2]{Ekeland-Temam76}. Now, we define the element $u \in L^\infty(I;\mo)$ by $u(t) = \gamma\text{sign}(\bar\varphi_i(x_t,t))\delta_{x_t}$. We have to check that $u:I \to \mo$ is weakly measurable. To this end the only delicate point is the weak measurability of $t \in I \mapsto \delta_{x_t} \in \mo$. This follows from the measurability of the mapping $t \mapsto x_t$ and the continuity of $x \in \bar\Omega \mapsto \delta_x \in \mo$ when $\mo$ is endowed with the weak$^*$ topology. By definition of $u$, the fact that $\bar\bu \in \uad$, and \eqref{E3.15} we get
\begin{align*}
&\gamma\int_0^T\|\bar\varphi_i(t)\|_{C_0(\omega)}\, dt = \int_0^T\langle u(t),\bar\varphi_i(t)\rangle_{\mo,C_0(\omega)}\, dt\\
&\le -\int_0^T\langle\bar u_i(t),\bar\varphi_i(t)\rangle_{\mo,C_0(\omega)}\, dt \le \gamma\int_0^T\|\bar\varphi_i(t)\|_{C_0(\omega)}\, dt.
\end{align*}
This implies
\[
-\int_0^T\langle\bar u_i(t),\bar\varphi_i(t)\rangle_{\mo,C_0(\omega)}\, dt = \gamma\int_0^T\|\bar\varphi_i(t)\|_{C_0(\omega)}\, dt
\]
and consequently
\begin{equation}
\int_0^T\langle u(t) + \bar u_i(t),\bar\varphi_i(t)\rangle_{\mo,C_0(\omega)}\, dt = 0.
\label{E3.16}
\end{equation}
Moreover, we have for almost every $t \in I$
\[
\langle u(t),\bar\varphi_i(t)\rangle_{\mo,C_0(\omega)} = \gamma\|\bar\varphi_i(t)\|_{C_0(\omega)} \ge -\langle \bar u_i(t),\bar\varphi_i(t)\rangle_{\mo,C_0(\omega)}.
\]
Whence we obtain $\langle u(t) + \bar u_i(t),\bar\varphi_i(t)\rangle_{\mo,C_0(\omega)} \ge 0$. This inequality along with \eqref{E3.16} yields
\[
-\langle\bar u_i(t),\bar\varphi_i(t)\rangle_{\mo,C_0(\omega)} = \langle u(t),\bar\varphi_i(t)\rangle_{\mo,C_0(\omega)} = \gamma\|\bar\varphi_i(t)\|_{C_0(\omega)}.
\]
This identity yields $\|\bar u_i(t)\|_\mo = \gamma$ if $\bar\varphi_i(t) \not\equiv 0$ and $-\langle\bar u_i(t),\bar\varphi_i(t)\rangle_{\mo,C_0(\omega)} = \|\bar u_i(t)\|_\mo\|\bar\varphi_i(t)\|_{C_0(\omega)}$ holds. Then, we can apply \cite[Lemma 3.4]{CCK2013} to get the inclusions \eqref{E3.14}.
\end{proof}

Next we define the Lagrangian function associated with the control problem \Pb. To this end, first we consider the functional $j:\mo \longrightarrow [0,\infty)$ given by $j(u) = \|u\|_\mo$. This is a convex and Lipschitz functional having directional derivatives $j'(u;v)$ for all $u, v \in \mo$. To give an expression for the derivative $j'(u;v)$ we consider the Lebesgue decomposition of $v$ with respect to $|u|$: $v = v_a + v_s$ with $dv_a = g_v d|u|$, where $v_a$ and $v_s$ are the absolutely continuous and singular parts of $v$ with respect to $|u|$, and $g_v \in L^1(|u|)$ is the Radon-Nikodym derivative of $v$ with respect to $|u|$. We can also write $du = g_u d|u|$ where $g_u$ is a measurable function such that $|g_u(x)| = 1$ for all $x \in \omega$. Actually, $g_u$ is the Radon-Nikodym derivative of $u$ with respect to $|u|$. The reader is referred, for instance, to \cite[Chapter 6]{Rudin70} for these issues. Now we have the following result taken from \cite[Proposition 3.3]{Casas-Kunisch2014}

\begin{proposition}
Let $u, v \in \mo$, then
\begin{equation}
j'(u;v) = \int_\omega g_v\, du + \|v_s\|_\mo.
\label{E3.17}
\end{equation}
\label{P3.1}
\end{proposition}

Given $\bu, \bv \in \lmo$, we denote by $g_{v_i}(t)$ the Radon-Nikodym derivative of $v_i(t)$ with respect to $|u_i(t)|$ and $v_{is}(t)$ the singular part of $v_i(t)$ with respect to $|u_i(t)|$. Then, $g_{v_i}:\omega \times I \longrightarrow \mathbb{R}$ is a measurable function.

Associated with the control problem \Pb we define the Lagrangian function
\[
\mL:\lmo \times L^1(I)^2 \longrightarrow \mathbb{R},\ \ \mL(\bu,\bpsi) = J(u) + \sum_{i = 1}^2\int_0^T\psi_i(t)j(u_i(t))\, dt.
\]
According to \eqref{E3.1} and \eqref{E3.17} the directional derivative of $\mL$ with respect to the first variable is given by
\begin{align}
\frac{\partial\mL}{\partial\bu}(\bu,\bpsi)\bv &= \int_0^T\langle\bv(t),\bvarphi_u(t)\rangle_{\bM,\bCo}\notag\\
&+ \sum_{i = 1}^2\int_0^T\psi_i(t)\left\{\int_\omega g_{v_i}(t)\, du_i(t) + \|v_{is}(t)\|_\mo\right\}\, dt\notag\\
& = \sum_{i = 1}^2\int_0^T\left\{\int_\omega\varphi_{ui}(t)g_{v_i}(t)\, d|u_i|(t) + \psi_i(t)\int_\omega g_{v_i}(t)\, du_i(t)\right\}\, dt\notag\\
&+ \sum_{i = 1}^2\int_0^T\left\{\langle v_{is}(t),\varphi_{ui}(t)\rangle_{\mo,C_0(\omega)} + \psi_i(t)\|v_{is}(t)\|_\mo\right\}\, dt.\label{E3.18}
\end{align}

Denote by $\bar\bu \in \uad$ a control with associated adjoint state $\bar\bvarphi \in \bV^{2,1}(0,T)$ satisfying \eqref{E3.14}. We define the function $\bar\bphi$ as follows
\[
\bar\phi_i(t) = \left\{\begin{array}{cl}1 & \text{if } \bar\varphi_i(t) \equiv 0,\\-\frac{\bar\varphi_i(t)}{\|\bar\varphi_i(t)\|_{C_0(\omega)}}& \text{if } \bar\varphi_i(t) \not\equiv 0,\end{array}\right. \ \text{ for } i = 1, 2.
\]
Then, we infer with \eqref{E3.14} that
\[
\bar\phi_i(x,t) = \left\{\begin{array}{cl} +1 & \text{if } x \in \textrm{Supp}(\bar u_i^+(t)),\\-1 & \text{if } x \in \textrm{Supp}(\bar u_i^-(t)),\end{array}\right. \ \text{ for } i = 1, 2,
\]
and, consequently, $d\bar u_i(t) = \bar\phi_i(t) d|\bar u_i(t)|$ for $i = 1, 2$. Using these identities and setting $\bu = \bar\bu$ and $\bar\psi_i(t) = \|\bar\varphi_i(t)\|_{C_0(\omega)}$, $i = 1, 2$, in \eqref{E3.18} we obtain the directional derivatives
\begin{align}
\hspace{-0.4cm}\frac{\partial\mL}{\partial\bu}(\bar\bu,\bar\bpsi)\bv &= \sum_{i = 1}^2\int_0^T\left\{\int_\omega\bar\varphi_i(t)g_{v_i}(t)\, d|\bar u_i|(t) + \|\bar\varphi_i(t)\|_{C_0(\omega)}\int_\omega g_{v_i}(t)\, d\bar u_i(t)\right\}\, dt\notag\\
&+ \sum_{i = 1}^2\int_0^T\left\{\langle v_{is}(t),\bar\varphi_i(t)\rangle_{\mo,C_0(\omega)} + \|\bar\varphi_i(t)\|_{C_0(\omega)}\|v_{is}(t)\|_\mo\right\}\, dt\notag\\
&= \sum_{i = 1}^2\int_0^T\left\{\int_\omega[\bar\varphi_i(t) + \|\bar\varphi_i(t)\|_{C_0(\omega)}\bar\phi_i(t)]g_{v_i}(t)\, d|\bar u_i|(t)\right.\, dt\notag\\
&+ \sum_{i = 1}^2\int_0^T\left\{\langle v_{is}(t),\bar\varphi_i(t)\rangle_{\mo,C_0(\omega)} + \|\bar\varphi_i(t)\|_{C_0(\omega)}\|v_{is}(t)\|_\mo\right\}\, dt\notag\\
&= \sum_{i = 1}^2\int_0^T\left\{\langle v_{is}(t),\bar\varphi_i(t)\rangle_{\mo,C_0(\omega)} + \|\bar\varphi_i(t)\|_{C_0(\omega)}\|v_{is}(t)\|_\mo\right\}\, dt.\label{E3.19}
\end{align}

From the above expression we deduce that $\frac{\partial\mL}{\partial\bu}(\bar\bu,\bar\bpsi)$ can be extended to a linear continuous form $\frac{\partial\mL}{\partial\bu}(\bar\bu,\bar\bpsi):\lmoq \longrightarrow \mathbb{R}$. Indeed, taking into account that $\bV^{1,2}(0,T) \subset L^2(I;\bCo) \subset L^{q'}(I;\bCo)$ we have
\[
|\frac{\partial\mL}{\partial\bu}(\bar\bu,\bar\bpsi)\bv| \le \|\bar\bvarphi\|_{L^{q'}(I;\bCo)}\|\bv\|_\lmoq\quad \forall \bv \in \lmoq.
\]

From the inequality
\[
|\langle v_{is}(t),\bar\varphi_i(t)\rangle_{\mo,C_0(\omega)}| \le \|\bar\varphi_i(t)\|_{C_0(\omega)}\|v_{is}(t)\|_\mo\ \text{ for } i= 1, 2,
\]
\cite[Lemma 3.4]{CCK2013}, the fact that $v_{is}(t)$ is singular with respect to $|\bar u_i(t)|$, and recalling that $\bar\bpsi(t) = \|\bar\bvarphi(t)\|_\bCo$ we deduce
\begin{align}
&\frac{\partial\mL}{\partial\bu}(\bar\bu,\bar\bpsi)\bv \ge 0\ \ \forall \bv \in \lmoq,\label{E3.20}\\
&\frac{\partial\mL}{\partial\bu}(\bar\bu,\bar\bpsi)\bv = 0 \ \text{ if and only if for } i = 1 \text{ and } 2, \text{ in case } \bar\varphi_i(t) \not\equiv 0\notag\\
&\hspace{0.25cm}\left\{\begin{array}{l}\textrm{Supp}(\bar v_{is}^+(t)) \subset \{x \in \omega\setminus\textrm{Supp}(|\bar u_{i}(t)|) : \bar\varphi_i(x,t) = -\|\bar\varphi_i(t)\|_{C_0(\omega)}\},\\
\textrm{Supp}(\bar v_{is}^-(t)) \subset \{x \in \omega\setminus\textrm{Supp}(|\bar u_{i}(t)|) : \bar\varphi_i(x,t) = +\|\bar\varphi_i(t)\|_{C_0(\omega)}\}.\end{array}\right. \label{E3.21}
\end{align}

\section{Second order optimality conditions}
\label{S4}
\setcounter{equation}{0}

In this section we study the second order necessary a sufficient optimality conditions for local optimality. Associated with $\bar\bu$ of \Pb we introduce the cone of critical directions
\begin{equation}
C_{\bar\bu} = \{\bv \in \lmoq : v_i \text{ satisfies } \eqref{E4.2}, i = 1,2, \text{ and }\frac{\partial\mL}{\partial\bu}(\bar\bu,\bar\bpsi)\bv = 0\}.
\label{E4.1}
\end{equation}

\begin{equation}
\text{For a.a. } t \in I: \text{ if } \|\bar u_i(t)\|_\mo = \gamma\ \text{then}\ \left\{\begin{array}{l}j'(\bar u_i(t);v_i(t)) \le 0,\\j'(\bar u_i(t);v_i(t)) = 0 \text{ if } \bar\varphi_i(t) \not\equiv 0.\end{array}\right.
\label{E4.2}
\end{equation}

Now, we formulate the second order necessary optimality condition.

\begin{theorem}
Let $\bar\bu$ be a local minimum of \Pb. Then, $J''(\bar\bu)\bv^2 \ge 0$ $\forall \bv \in C_{\bar\bu}$ holds.
\label{T4.1}
\end{theorem}

\begin{proof}
Let us take $\bv \in C_{\bar\bu} \cap \lmo$. We set
\[
v_i(t) = g_{v_i}(t)d|\bar u_i(t)| + v_{is}(t)\ \text{ and }\ d\bar u_i(t) = g_{\bar u_i}(t) d|\bar u_i(t)|, i = 1, 2,
\]
where $g_{v_i}(t)$ and $g_{\bar u_i}(t)$ are the corresponding Randon-Nikodym derivatives. We define the sets
\begin{align*}
&I^0_{\gamma,i} = \{t \in I : \|\bar u_i(t)\|_\mo = \gamma\text{ and } \bar\varphi_i(t) \equiv 0\},\\
& I^+_{\gamma,i} = \{t \in I : \bar\varphi_i(t) \not\equiv 0\},\ I_{\gamma,i} = I^0_{\gamma,i} \cup I^+_{\gamma,i}, \ \ i = 1, 2.
\end{align*}
Note further that $I = I_{\gamma,i} \cup \{t \in I : \|\bar u(t)\|_\mo < \gamma\}$. From \eqref{E3.14} it follows that $\|\bar u_i(t)\|_\mo = \gamma$ for every $t \in I_{\gamma,i}$. Proposition \ref{P3.1} and $\bv \in C_{\bar\bu}$ yield
\begin{equation}
j'(\bar u_i(t),v_i(t)) = \int_\omega g_{v_i}(t)\,d\bar u_i(t) + \|v_{is}(t)\|_\mo\left\{\begin{array}{l} = 0 \text{ if } t \in I^+_{\gamma,i},\\ \le 0 \text{ if } t\in I^0_{\gamma,i},\end{array}\right. \ i= 1, 2.
\label{E4.3}
\end{equation}
Let us denote
\[
a(t) = \int_\omega g_{v_i}(t)\,d\bar u_i(t)\ \text{ and }\ a_k(t) = \int_\omega \proj_{[-k,+k]}(g_{v_i}(t))\,d\bar u_i(t)\ \text{ for } k \ge 1.
\]
With \eqref{E4.3} and Lebesgue's theorem we infer
\begin{equation}
j'(\bar u_i(t);v_i(t)) = a(t)+ \|v_{is}(t)\|_\mo\left\{\begin{array}{l} = 0 \text{ if } t \in I^+_{\gamma,i},\\ \le 0 \text{ if } t\in I^0_{\gamma,i},\end{array}\right. \text{and}\, \lim_{k \to \infty} a_k(t) = a(t),\label{E4.4}
\end{equation}
in the a.e.~sense. Now, we set
\[
g^k_{v_i}(t) = \proj_{[-k,+k]}(g_{v_i}(t)) + \frac{a(t) - a_k(t)}{\gamma}g_{\bar u_i}(t)
\]
and
\[
dv_{k,i}(t) = \left\{\begin{array}{cl}0 & \text{if } \displaystyle \gamma - \frac{1}{k} < \|\bar u_i(t)\|_\mo < \gamma,\\g^k_{v_i}(t)d\bar u_i(t) + dv_{is}(t) & \text{if } t \in I_{\gamma,i},\\ dv_i(t) & \text{otherwise}.\end{array}\right.
\]
Below we shall argue that $\bv_k \to \bv$ in $L^q(I;\bM)$. From \eqref{E4.4} we get
\begin{align*}
&j'(\bar u_i(t);v_{k,i}(t)) = a_k(t) + \frac{a(t) - a_k(t)}{\gamma}\int_\omega g_{\bar u_i}(t)\, d\bar u_i(t) + \|v_{is}(t)\|_\mo\\
& = a_k(t) + \frac{a(t) - a_k(t)}{\gamma}\int_\omega d|\bar u_i(t)| + \|v_{is}(t)\|_\mo\\
& = a(t) + \|v_{is}(t)\|_\mo \left\{\begin{array}{l} = 0 \text{ if } t \in I^+_{\gamma,i},\\ \le 0 \text{ if } t\in I^0_{\gamma,i}.\end{array}\right.
\end{align*}
Moreover, from \eqref{E3.19} we have that $\frac{\partial\mL}{\partial\bu}(\bar\bu,\bar\bpsi)\bv$ only depends on the singular part of $\bv(t)$ with respect to $\bar\bu(t)$. Since the singular part of $\bv_k(t)$ is zero or equal to the singular part of $\bv(t)$, we conclude that $\frac{\partial\mL}{\partial\bu}(\bar\bu,\bar\bpsi)\bv_k = 0$. This identity along with $j'(\bar u_i(t);v_{k,i}(t)) = 0$ for $t \in I_{\gamma,i}^+$, $\bar\varphi_i(t) \equiv 0$ on $I\setminus I^+_{\gamma,i}$, and the equality
\[
\frac{\partial\mL}{\partial\bu}(\bar\bu,\bar\bpsi)\bv_k = J'(\bar\bu)\bv_k  + \sum_{i = 1}^2\int_0^T\|\bar\varphi_i(t)\|_{C_0(\omega)}j'(\bar u_i(t);v_{k,i}(t))\, dt
\]
imply that $J'(\bar\bu)\bv_k = 0$.

Next we prove that $\bar\bu + \rho\bv_k \in \uad$ for every $\rho > 0$ small enough. Indeed, first we observe that
\[
|g^k_{v,i}(t)| \le k + \frac{|a(t) - a_k(t)|}{\gamma} \le k + \frac{2}{\gamma}\int_\omega|g_{v_i}(t)|\, d|\bar u_i(t)| \le k + \frac{2}{\gamma}\|\bv\|_\lmo.
\]
Let us take $\rho_k > 0$ such that
\[
\rho_k\Big(1 + \frac{2}{\gamma}\Big)\|\bv\|_\lmo < \frac{1}{k}.
\]
For $i = 1, 2$ , using \eqref{E4.4}, we deduce for $t \in I_{\gamma,i}$ and $\rho \le \rho_k$
\begin{align*}
&\|\bar u_i(t) + \rho v_{k,i}(t)\|_\mo = \int_\omega|g_{\bar u_i}(t) + \rho g^k_{v_i}(t)|\, d|\bar u_i(t)| + \rho\|v_{is}(t)\|_\mo\\
&= \int_\omega(1 + \rho g^k_{v_i}(t))\, d\bar u_i^+ + \int_\omega(1 - \rho g^k_{v_i}(t))\, d\bar u_i^- + \rho\|v_{is}(t)\|_\mo\\
& = \gamma + \rho\Big(\int_\omega g^k_{v_i}(t)\, d\bar u_i + \|v_{is}(t)\|_\mo\Big) = \gamma + \rho\Big(a(t) + \|v_{is}(t)\|_\mo\Big) \le \gamma.
\end{align*}
Moreover, if $\|\bar u(t)\|_\mo \le \gamma - \frac{1}{k}$, then $\|\bar u(t) + \rho\bv_k(t)\|_\mo = \|\bar u(t) + \rho\bv(t)\|_\mo\le \gamma$ $\forall \rho < \rho_k$ holds. If $\gamma - \frac{1}{k} < \|\bar u(t)\|_\mo < \gamma$, then $\|\bar u(t) + \rho\bv_k(t)\|_\mo = \|\bar u(t)\|_\mo < \gamma$ is fulfilled. Thus, we have that $\bar\bu + \rho\bv_k \in \uad$ for every $\rho < \rho_k$.

Now, using that $\bar\bu$ is a local minimum of \Pb, $J'(\bar\bu)\bv_k = 0$ as proved before, and performing a Taylor expansion we get for $k$ fixed and $\rho$ small enough
\[
0 \le J(\bar\bu + \rho\bv_k)) - J(\bar\bu) = \rho J'(\bar\bu)\bv_k + \frac{\rho^2}{2}J''(\bar\bu + \theta\rho\bv_k)\bv_k^2 = \frac{\rho^2}{2}J''(\bar\bu + \theta\rho\bv_k)\bv_k^2.
\]
Dividing the expression by $\rho^2/2$, using the fact that $J:L^q(I;\bM) \to \mathbb{R}$ is of class $C^\infty$, and taking $\rho\to 0$ we infer $J''(\bar\bu)\bv_k^2 \ge 0$. Now, using again Lebesgue's theorem it follows that for almost every $t \in I$
\begin{align*}
&\lim_{k \to \infty}\|g^k_{v_i}(t) - g_{v_i}(t)\|_{L^1(|\bar u_i(t)|)} = 0\ \text{ and }\\
&\|g^k_{v_i}(t) - g_{v_i}(t)\|_{L^1(|\bar u_i(t)|)} \le 2 \|g_{v_i}(t)\|_{L^1(|\bar u_i(t)|)} \le 2\|v_i\|_\lmo
\end{align*}
Using these properties we easily obtain that $\bv_k \to \bv$ in $L^q(I;\bM)$. Then, with Theorem \ref{T3.2} we can pass to the limit when $k \to \infty$ in the above inequality and conclude that $J''(\bar\bu)\bv^2 \ge 0$.

Finally, if $\bv \in C_{\bar\bu} \setminus \lmo$, then we take $\{\bv_k\}_{k = 1}^\infty \subset \lmo$ defined as follows
\[
\bv_k(t) = \left\{\begin{array}{cl}0 &\text{if } \|\bv_k(t\|_\mo > k,\\\bv(t) &\text{otherwise.}\end{array}\right.
\]
It is straightforward to check $\bv_k \in C_{\bar\bu} \cap \lmo$ for every $k \ge 1$ and $\bv_k \to \bv$ in $\lmoq$. Hence, $J''(\bar\bu)\bv_k^2 \ge 0$ holds for every $k$, and passing to the limit we obtain $J''(\bar\bu)\bv^2 \ge 0$.
\end{proof}

In order to formulate a second order sufficient condition for local optimality we need to extend the cone of critical directions. Given $(\bar\bu,\bar\bvarphi) \in \uad \times \bV^{2,1}(0,T)$ satisfying \eqref{E3.13}--\eqref{E3.14}, we define for $\tau > 0$
\begin{align}
C^\tau_{\bar\bu} = &\{\bv \in \lmoq : v_i \text{ satisfies } \eqref{E4.6}, i = 1,2, \text{ and}\notag\\
&\frac{\partial\mL}{\partial\bu}(\bar\bu,\bar\bpsi)\bv \le \tau\|\bz_\bv\|_\bLqd\},
\label{E4.5}
\end{align}
where $\bz_\bv = G'(\bar\bu)\bv$.

\begin{equation}
\left\{\begin{array}{l}\text{For a.a. } t \in I: \text{ if } \|\bar u_i(t)\|_\mo = \gamma \text{ \rm then } j'(\bar u_i(t);v_i(t)) \le 0,\\
\text{moreover }\sum_{i = 1}^2\int_0^T\|\bar\varphi_i(t)\|_{C_0(\omega)}j'(\bar u_i(t);v_i(t))\, dt \ge -\tau\|\bz_\bv\|_\bLqd.\end{array}\right.
\label{E4.6}
\end{equation}

The last condition is a relaxation of the second condition of \eqref{E4.2}.

\begin{theorem}
Let $(\bar\bu,\bar\bvarphi) \in \uad \times \bV^{2,1}(0,T)$ satisfy \eqref{E3.13}--\eqref{E3.14}. Assume that
\begin{align}
&\by_0 \in \bBqp \text{ and } \exists r \in (4,q] \text{ such that } \by_d, \bar\by \in L^r(I;\bLf), \label{E4.7}\\
&\exists\tau > 0 \text{ and } \delta > 0 \text{ such that } J''(\bar\bu)\bv^2 \ge \delta\|\bz_\bv\|^2_\bLqd\ \ \forall \bv \in C^\tau_{\bar\bu}.
\label{E4.8}
\end{align}
Then, there exist $\kappa > 0$ and $\varepsilon > 0$ such that
\begin{equation}
J(\bar\bu) + \frac{\kappa}{2}\|\by_\bu - \bar\by\|^2_\bLqd \le J(\bu)\ \ \forall \bu \in \uad : \|\bu - \bar\bu\|_{L^q(I;\bWmop)} \le \varepsilon,
\label{E4.9}
\end{equation}
where $\bar\by = G(\bar\bu)$.
\label{T4.2}
\end{theorem}

\begin{remark}
Notice that in the proof of Theorem \ref{T3.2} the continuous embedding $\mY \subset L^4(I;\bLf))$ was established. Hence, the assumption $\bar\by \in L^r(I;\bLf)$ for some $r > 4$ is not too restrictive. Actually, we think that this regularity is enjoyed by the solutions of the state equation, but we have not been able to prove it.
\label{R4.1}
\end{remark}

In order to prove this theorem we need to establish some lemmas.

\begin{lemma}
There exists a constant $M_\gamma$ such that
\begin{equation}
\|\by_\bu - \bar\by\|_\mY \le M_\gamma\|\bu - \bar\bu\|_{L^q(I;\bWmop)}\ \ \forall \bu \in \uad.
\label{E4.10}
\end{equation}
\label{L4.1}
\end{lemma}

\begin{proof}
Let $G_0:L^q(I;\bWmop) \longrightarrow \mY$ be as defined in the proof of Theorem \ref{T2.2}. Then, from mean value theorem we infer
\begin{align*}
\|\by_\bu - \bar\by\|_\mY &= \|G_0(\chi_\omega\bu) - G_0(\chi_\omega\bar\bu)\|_\mY\\
&\le \sup_{\bv \in \uad}\|G_0'(\chi_\omega\bv)\|_{\mL(L^q(I;\bWmop),\mY)}\|\bu - \bar\bu\|_{L^q(I;\bWmop)}\\
&= M_\gamma\|\bu - \bar\bu\|_{L^q(I;\bWmop)}.
\end{align*}
The constant $M_\gamma$ is finite; see the proof of \cite[Theorem 5.1]{Casas-Kunisch2020}.
\end{proof}

\begin{lemma}
Given $\bu \in \uad$ and $\bv \in \lmoq$, we set $\bz_{\bu,\bv} = G'(\bu)\bv$ and $\bz_\bv = G'(\bar\bu)\bv$. Then, there exist constants $M_1 > 0$ and $M_2 > 0$ independent of $\bu$ and $\bv$ such that
\begin{align}
&\|\bz_{\bu,\bv} - \bz_\bv\|_\bLqd \le M_1\|\bu - \bar\bu\|_{L^q(I;\bWmop)}\|\bz_\bv\|_\bLqd,
\label{E4.11}\\
&\|\bz_{\bu,\bv}\|_\bLqd \le M_2\|\bz_\bv\|_\bLqd.
\label{E4.12}
\end{align}
\label{L4.2}
\end{lemma}

\begin{proof}
According to \eqref{E2.3}, the equations satisfied by $\bz_{\bu,\bv}$ and $\bz_\bv$ are
\begin{align*}
&\frac{\partial\bz_{\bu,\bv}}{\partial t} - \nu\Delta\bz_{\bu,\bv} + (\by_\bu\cdot \bna)\bz_{\bu,\bv} + (\bz_{\bu,\bv} \cdot\bna)\by_\bu + \nabla\q_\bu = \chi_\omega\bv,\\
&\frac{\partial\bz_\bv}{\partial t}-\nu\Delta\bz_\bv + (\bar\by\cdot \bna)\bz_\bv + (\bz_\bv \cdot\bna)\bar\by + \nabla\bar\q = \chi_\omega\bv.
\end{align*}
Subtracting both equations and setting $\be = \bz_{\bu,\bv} - \bz_\bv$ and $\q = \q_\bv - \bar\q$ we get
\[
\left\{\begin{array}{l}\displaystyle\frac{\partial\be}{\partial t}-\nu\Delta\be + (\by_\bu\cdot \bna)\be + (\be \cdot\bna)\by_\bu + \nabla\q = \bg \ \text{ in } Q,\\[1.2ex]\div\be = 0 \ \text{ in } Q, \ \be = 0 \ \text{ on } \Sigma,\ \be(0) = 0 \text{ in } \Omega\end{array}\right.
\]
where $\bg = -[(\by_\bu - \bar\by)\cdot\bna]\bz_\bv - (\bz_\bv\cdot \bna)(\by_\bu - \bar\by)$. From \cite[Lemma 2.1]{Casas-Kunisch2020} we get that $\bg \in L^2(I;\bHmo)$. Then, \cite[Theorem 2.7]{Casas-Kunisch2020} implies that \eqref{E4.8} has a unique solution $(\be,\q) \in \bWoT \times W^{-1,\infty}(I;L^2(\Omega)/\mathbb{R})$. Take $\bef \in \bLqd$ arbitrary and let $\bvarphi \in \bV^{2,1}(0,T)$ be the solution of the adjoint state equation \eqref{E3.3} with $\by_\bu - \by_d$ replaced by $\bef$. We have the estimate
\begin{equation}
\|\bvarphi\|_{\bV^{2,1}(0,T)} \le C\|\bef\|_\bLqd\quad \forall \bef \in \bLqd,\ \forall \bu \in \uad.
\label{E4.13}
\end{equation}
Then, we have
\begin{align*}
&\int_0^T\int_\Omega\bef \be\, dx\, dt = \int_0^T\int_\Omega\Big[-\frac{\partial\bvarphi}{\partial t} - \nu\Delta\bvarphi - (\by_\bu \cdot \bna)\bvarphi - (\bna\bvarphi)^T\by_\bu + \nabla \pi\Big]\be\, dx\, dt\\
&= \int_0^T\Big[\frac{d}{dt}(\be,\bvarphi)_\bLd + a(\be,\bvarphi) + b(\by_\bu,\be,\bvarphi) + b(\be(t),\by_\bu,\bvarphi)\Big]\, dt\\
&= \int_0^T\langle\bg,\bvarphi\rangle_{\bHmo,\bHo}\, dt = -\int_0^T\big\{b(\by_\bu - \bar\by,\bz_\bv,\bvarphi) + b(\bz_\bv,\by_\bu - \bar\by,\bvarphi)\big\}\, dt.
\end{align*}
Let us estimate the last integral. To this end we use the embeddings $\bHto \subset L^4(I;\mathbf{W}^{1,4}(\Omega))$ and $\mY \subset L^4(I;\bLf)$, and estimates \eqref{E4.10} and \eqref{E4.13}:
\begin{align*}
&\int_0^T|b(\by_\bu - \bar\by,\bz_\bv,\bvarphi) |\, dt = \int_0^T|b(\by_\bu - \bar\by,\bvarphi,\bz_\bv) |\, dt\\
&\le \|\by_\bu - \bar\by\|_{L^4(I;\bLf)}\|\bvarphi\|_{L^4(I;\mathbf{W}^{1,4}(\Omega))}\|\bz_\bv\|_\bLqd\\
&\le C'\|\bef\|_\bLqd\|\by_\bu - \bar\by\|_\mY\|\bz_\bv\|_\bLqd\\
& \le C''M_\gamma \|\bef\|_\bLqd\|\bu - \bar\bu\|_{L^q(I;\bWmop)}\|\bz_\bv\|_\bLqd.
\end{align*}
The term $b(\bz_\bv,\by_\bu - \bar\by,\bvarphi)$ is estimated in the same way. Thus, we have
\[
\int_0^T\int_\Omega\bef \be\, dx\, dt \le M_1\|\bef\|_\bLqd\|\bu - \bar\bu\|_{L^q(I;\bWmop)}\|\bz_\bv\|_\bLqd
\]
$\forall \bef \in L^2(Q)$ and, consequently, \eqref{E4.11} is fulfilled. Finally, \eqref{E4.12} follows from \eqref{E4.11} and  the inequality
\[
\|\bz_{\bu,\bv}\|_\bLqd \le \|\bz_{\bu,\bv} - \bz_\bv\|_\bLqd + \|\bz_\bv\|_\bLqd.
\]
\end{proof}

\begin{lemma}
There exists $\varepsilon_0 > 0$ such that $\forall \bu \in \uad$ with $\|\bu - \bar\bu\|_{L^q(I;\bWop)} \le \varepsilon_0$ the inequality
\begin{equation}
\|\by_\bu - \bar\by\|_\bLqd \le 2\|\bz_{\bu - \bar\bu}\|_\bLqd
\label{E4.14}
\end{equation}
holds, where $\bz_{\bu - \bar\bu} = G'(\bar\bu)(\bu - \bar\bu)$.
\label{L4.3}
\end{lemma}

\begin{proof}
Let us consider the equations satisfied by $\by_\bu$, $\bar\by$ and $\bz_{\bu - \bar\bu}$:
\begin{align*}
&\frac{\partial\by_\bu}{\partial t} - \nu\Delta\by_\bu + (\by_\bu \cdot \bna)\by_\bu + \nabla\p_\bu = \chi_\omega\bu,\\
&\frac{\partial\bar\by}{\partial t} - \nu\Delta\bar\by + (\bar\by \cdot \bna)\bar\by + \nabla\bar\p = \chi_\omega\bar\bu,\\
&\frac{\partial\bz_{\bu - \bar\bu}}{\partial t}-\nu\Delta\bz_{\bu - \bar\bu} + (\bar\by \cdot \bna)\bz_{\bu - \bar\bu} + (\bz_{\bu - \bar\bu}\cdot \bna)\bar\by + \nabla\q_{\bu - \bar\bu} = \chi_\omega(\bu - \bar\bu).
\end{align*}
Setting $\be = \by_\bu - \bar\by - \bz_{\bu - \bar\bu}$ and $\q = \p_u - \bar\p - \q_{\bu - \bar\bu}$, we infer from the above equations
\[
\frac{\partial\be}{\partial t} - \nu\Delta\be + (\bar\by \cdot \bna)\be + (\be\cdot \bna)\bar\by + \nabla\q = -[(\by_\bu - \bar\by)\cdot\bna](\by_\bu - \bar\by).
\]
Using again \cite[Lemma 2.1]{Casas-Kunisch2020}, we have that $[(\by_\bu - \bar\by)\cdot\bna](\by_\bu - \bar\by) \in L^2(I;\bHmo)$ and, hence, $\be \in \bWoT$. Arguing as in the proof of Lemma \ref{L4.2} and using \eqref{E4.10} we infer
\begin{align*}
&\|\be\|_\bLqd \le C_1\|\by_\bu - \bar\by\|_\mY\|\by_\bu - \bar\by\|_\bLqd\\ &\le C_2\|\bu - \bar\bu\|_{L^q(I;\bWmop)}\|\by_\bu - \bar\by\|_\bLqd.
\end{align*}
Let us take $0 < \varepsilon_0 < \frac{1}{2C_2}$. Then, we have
\begin{align*}
&\|\by_\bu - \bar\by\|_\bLqd \le \|\be\|_\bLqd + \|\bz_{\bu - \bar\bu}\|_\bLqd\\
&\le \frac{1}{2}\|\by_\bu - \bar\by\|_\bLqd + \|\bz_{\bu - \bar\bu}\|_\bLqd,
\end{align*}
which implies \eqref{E4.14}.
\end{proof}

\begin{lemma}
Assume that \eqref{E4.7} holds. Then, there exists $\bar\varepsilon > 0$ such that $\by_\bu \in \mY \cap L^r(I;\bLf)$ for every $ \bu \in B_{\bar\varepsilon}(\bar\bu) \subset \lmoq$.
Moreover, if $\{\bu_k\}_{k = 1}^\infty \subset B_{\bar\varepsilon}(\bar\bu)$ is a sequence converging to $\bar\bu$ in $L^q(I;\bWmop)$, then $\by_{\bu_k} \to \bar\by$ in $L^r(I;\bLf)$ holds.
\label{L4.4}
\end{lemma}

\begin{proof}
The proof is split in three steps.

{\em Step I-} From \cite[Theorem 2.5]{Casas-Kunisch2020} we know that the system
\begin{equation}
\left\{\begin{array}{l}\displaystyle\frac{\partial \by_S}{\partial t} -\nu\Delta\by_S +\nabla\p_S = \bef_0 + \bu\chi_\omega\ \text{ in } Q,\\[1.2ex]\div\by_S = 0 \ \text{ in } Q, \ \by_S = 0 \ \text{ on } \Sigma,\ \by_S(0) = \by_0 \text{ in } \Omega\end{array}\right.
\label{E4.15}
\end{equation}
has a unique solution $\by_S \in \bWqpoT$ satisfying
\begin{equation}
\|\by_S\|_\bWqpoT \le C_1\Big(\|\bef_0\|_{L^q(I;\bWpc')} + \|\bu\|_{L^q(I;\bWmop)} + \|\by_0\|_{\bBqp}\Big)
\label{E4.16}
\end{equation}
for some constant $C_1$ independent of $\bu$. Since $p \ge \frac{4}{3}$ and $r \le q$, we have that $\by_S \in L^r(I;\bLf)$.

Now, we take $\by \in \bWoT$ as the solution of
\begin{equation}
\left\{\begin{array}{l}\displaystyle\frac{\partial \by}{\partial t} -\nu\Delta\by + (\by \cdot \bna)\by + (\by_S \cdot \bna)\by + (\by \cdot \bna)\by_S + \nabla\p = -(\by_S\cdot\nabla)\by_S\ \text{ in } Q,\\[1.2ex]\div\by = 0 \ \text{ in } Q, \ \by = 0 \ \text{ on } \Sigma,\ \by(0) = 0 \text{ in } \Omega.\end{array}\right.
\label{E4.17}
\end{equation}
The existence and uniqueness of $\by$ follows from \cite[Theorem 2.7]{Casas-Kunisch2020}, as well as the estimate
\begin{equation}
\|\by\|_\bWoT \le \hat\eta\Big(\|\by_S\|_{L^q(I;\bWp)}\Big),
\label{E4.18}
\end{equation}
where $\hat\eta:[0,\infty) \longrightarrow [0,\infty)$ is a nondecreasing function with $\hat\eta(0) = 0$. Obviously, the solution of \eqref{E1.1} is given by $\by_\bu = \by_S + \by$. In the sequel, applying  the implicit function theorem, we will prove that $\by \in L^r(I;\bLf)$ if $\bu \in B_{\bar\varepsilon}(\bar\bu)$ for some $\bar\varepsilon > 0$.

{\em Step II-} First, we write $\bar\by = \tilde\by_S + \tilde\by$ with $\tilde\by_S$ and $\tilde\by$ solutions of \eqref{E4.15} and \eqref{E4.16} with $\bu$ and $\by_S$ replaced by $\bar\bu$ and $\tilde\by_S$, respectively. Let us prove that $\tilde\by \in \mathbf{W}_{\frac{r}{2},2}(0,T)$. Observe that $\tilde\by$ satisfies the Stokes equations
\[
\frac{\partial \tilde\by}{\partial t} -\nu\Delta\tilde\by + \nabla\tilde\p = \bg \ \text{ in } Q,
\]
where $\bg =  -(\tilde\by_S \cdot \nabla)\tilde\by_S - (\tilde\by \cdot \bna)\tilde\by - (\tilde\by_S \cdot \bna)\tilde\by - (\tilde\by \cdot \bna)\tilde\by_S$. Then, using the maximal parabolic regularity for the Stokes system, it is enough to prove that $\bg \in L^{\frac{r}{2}}(I;\bHmo)$ to deduce that $\tilde\by \in \mathbf{W}_{\frac{r}{2},2}(0,T)$. First we observe that \eqref{E4.7} implies that $\tilde\by = \bar\by - \tilde\by_S \in L^r(I;\bLf)$. Let us prove $(\tilde\by_S \cdot \nabla)\tilde\by_S \in L^{\frac{r}{2}}(I;\bHmo)$. Indeed, given $\bz \in \bHo$ we have
\[
|\langle(\tilde\by_S(t) \cdot \nabla)\tilde\by_S(t),\bz\rangle| = |\langle(\tilde\by_S(t) \cdot \nabla)\bz,\tilde\by_S(t)\rangle|\le \|\tilde\by_S(t)\|^2_\bLf\|\bz\|_\bHo.
\]
Then, we have $\|(\tilde\by_S \cdot \nabla)\tilde\by_S\|_{L^{\frac{r}{2}}(I;\bHmo)} \le \|\tilde\by_S\|^2_{L^r(I;\bLf)}$. In a similar way we get that $\|(\tilde\by \cdot \nabla)\tilde\by\|_{L^{\frac{r}{2}}(I;\bHmo)} \le \|\tilde\by\|^2_{L^r(I;\bLf)}$ and $\|(\tilde\by \cdot \nabla)\tilde\by_S\|_{L^{\frac{r}{2}}(I;\bHmo)} = \|(\tilde\by_S \cdot \nabla)\tilde\by\|_{L^{\frac{r}{2}}(I;\bHmo)} \le \|\tilde\by\|_{L^r(I;\bLf)}\|\tilde\by_S\|_{L^r(I;\bLf)}$. All together this leads to
\[
\|\tilde\by\|_{\mathbf{W}_{\frac{r}{2},2}(0,T)} \le C_2\Big(\|\tilde\by\|_{L^r(I;\bLf)} + \|\tilde\by_S\|_{L^r(I;\bLf)}\Big)^2.
\]

{\em Step III-} We define the mapping
\begin{align*}
&\mF:\mathbf{W}_{\frac{r}{2},2}(0,T) \times L^q(I;\bWmop) \longrightarrow L^{\frac{r}{2}}(I;\bV')\\
&\mF(\by,\bu) = \frac{\partial \by}{\partial t} + A\by + (\by \cdot \bna)\by + (\by_S(\bu) \cdot \bna)\by + (\by \cdot \bna)\by_S(\bu) + (\by_S(\bu)\cdot\nabla)\by_S(\bu),
\end{align*}
where $\by_S(\bu)$ is the solution \eqref{E4.15} and $A:\bV \longrightarrow \bV'$ is given by $\langle A\by,\bz\rangle_{\bV',\bV} = a(\by,\bz)$. Using \cite[Theorem 3]{Amann2001} with $\bX_0 = \bHmo$, $\bX_1 = \bHo$, $p = \frac{r}{2}$, $s = \frac{1}{r}$, and $\theta = \frac{3}{4}$ we obtain
\begin{align*}
&\mathbf{W}_{\frac{r}{2},2}(0,T) \subset L^r(I;(\bHmo,\bHo)_{\frac{3}{4},1}) \subset L^r(I;(\bHmo,\bHo)_{\frac{3}{4},2})\\
& = L^r(I;\bH^{\frac{1}{2}}(\Omega)) \subset L^r(I;\bLf).
\end{align*}
Arguing as in {\em Step II}, it yields $(\by \cdot \bna)\by + (\by_S(\bu) \cdot \bna)\by + (\by \cdot \bna)\by_S(\bu) + (\by_S(\bu)\cdot\nabla)\by_S(\bu) \in L^{\frac{r}{2}}(I;\bV')$ for every $\by \in \mathbf{W}_{\frac{r}{2},2}(0,T)$. Consequently, $\mF$ is well defined. Furthermore, it is a $C^\infty$ function. We have that $\mF(\tilde\by,\bar\bu) = 0$. Moreover, the partial derivative
\begin{align*}
&\frac{\partial\mF}{\partial\by}(\tilde\by,\bar\bu):\mathbf{W}_{\frac{r}{2},2}(0,T) \longrightarrow L^{\frac{r}{2}}(I;\bV')\\
&\frac{\partial\mF}{\partial\by}(\tilde\by,\bar\bu)\bz = \frac{\partial \bz}{\partial t} + A\bz + (\bar\by \cdot \bna)\bz + (\bz \cdot \bna)\bar\by,
\end{align*}
where $\bar\by = \tilde\by + \tilde\by_S = \tilde\by + \by_S(\bar\bu)$, is an isomorphism. Indeed, the injectivity follows from \cite[Theorem 2.7]{Casas-Kunisch2020}. Let us prove the surjectivity. Given $\bef \in L^{\frac{r}{2}}(I;\bV')$, we take a sequence $\{\bef_k\}_{k = 1}^\infty \subset L^{\frac{r}{2}}(I;\bH)$ such that $\bef_k \to \bef$ in $L^{\frac{r}{2}}(I;\bV')$. For every $k$ we consider the equation
\[
\left\{\begin{array}{l}\displaystyle\frac{\partial \bz_k}{\partial t} + A\bz_k + (\bar\by \cdot \bna)\bz_k + (\bz_k \cdot \bna)\bar\by = \bef_k \ \text{ for a.a. } t \in I,\\\bz_k(0) = 0.\end{array}\right.
\]
Arguing as we did for equation \eqref{E3.3}, we get that $\bz_k \in \bV^{2,1}(0,T)$. Moreover, using again \cite[Theorem 2.7]{Casas-Kunisch2020}, we have the estimate analogous to \eqref{E4.18} for $k_0$ large enough:
\begin{equation}
\|\bz_k\|_\bWoT \le \hat\eta\Big(\|\bef_k\|_{L^{\frac{r}{2}}(I;\bV')}\Big) \le \hat\eta\Big(2\|\bef\|_{L^{\frac{r}{2}}(I;\bV')}\Big)\ \ \forall k \ge k_0.
\label{E4.19}
\end{equation}
Observe that $\bz_k$ satisfies the Stokes equations
\[
\frac{\partial \bz_k}{\partial t} + A\bz_k = \bg_k ,
\]
where $\bg_k =  \bef_k - (\bar\by \cdot \bna)\bz_k - (\bz_k \cdot \bna)\bar\by$. Then, using again the maximal parabolic regularity for the Stokes system we have
\[
\|\bz_k\|_{\mathbf{W}_{\frac{r}{2},2}(0,T)} \le C_3\|\bg_k\|_{L^{\frac{r}{2}}(I;\bV')} \le C_3\Big(\|\bef_k\|_{L^{\frac{r}{2}}(I;\bV')} + 2\|\bar\by\|_{L^r(I;\bLf)}\|\bz_k\|_{L^r(I;\bLf)}\Big).
\]
From \cite[Theorem 3]{Amann2001}, we know that the embedding $\mathbf{W}_{\frac{r}{2},2}(0,T) \subset L^r(I;\bLf)$ is compact. Then, we can apply Lions's Lemma with $\mathbf{W}_{\frac{r}{2},2}(0,T) \subset L^r(I;\bLf) \subset \bLqd$ to deduce the existence of a constant $C_4$ such that
\[
\|\bz_k\|_{L^r(I;\bLf)} \le \frac{1}{4C_3\|\bar\by\|_{L^r(I;\bLf)}}\|\bz_k\|_{\mathbf{W}_{\frac{r}{2},2}(0,T)} + C_4\|\bz_k\|_\bLqd.
\]
The last two inequalities and \eqref{E4.19} imply that $\{\bz_k\}_{k = 1}^\infty$ is bounded in $\mathbf{W}_{\frac{r}{2},2}(0,T)$. Then, taking a subsequence,  we have that $\bz_k \rightharpoonup \bz$ in $\mathbf{W}_{\frac{r}{2},2}(0,T)$ with $\frac{\partial\mF}{\partial\by}(\tilde\by,\bar\bu)\bz = \bef$, which proves the surjectivity. Hence, from the implicit function theorem we conclude the existence of $\bar\varepsilon > 0$ such that the statement of the lemma is fulfilled.
\end{proof}

\begin{lemma}
Assume that \eqref{E4.7} holds and let $\bar\varepsilon$ be as defined in Lemma \ref{L4.4}. Then, for every $\bu \in B_{\bar\varepsilon}(\bar\bu)$ the solution $\bvarphi_\bu$ of \eqref{E3.3} belongs to $C(\bar I;\mathbf{C}^1(\bar\Omega))$ and there exists a constant $M_3$ continuously depending on $\|\by_\bu\|_{L^r(I;\bLf)}$ such that
\begin{equation}
\|\bvarphi_\bu\|_{C(\bar I;\mathbf{C}^1(\bar\Omega))} \le M_3\|\by_\bu - \by_d\|_{L^r(I;\bLf)}.
\label{E4.20}
\end{equation}
\label{L4.5}
\end{lemma}

\begin{proof}
Let us consider the spaces
\begin{align*}
& \bX = \{\by \in L^r(I;\mathbf{W}^{2,4}(\Omega)) \cap W^{1,r}(I;\bLf) : \by = 0 \text{ on } \Sigma \text{ and } \div\by = 0 \text{ in } Q\},\\
&\Pi = \{\pi \in L^r(I;W^{1,4}(\Omega)) : \int_\Omega\pi(t) \, dx = 0 \text{ for a.a. } t \in I\}.
\end{align*}
Applying \cite[Theorem 3]{Amann2001} with $\bX_0 = \bLf$, $\bX_1 = \mathbf{W}^{2,4}(\Omega)$, $p = r$, $\frac{1}{r} < s < \frac{1}{4}$, and $\frac{3}{4} < \theta < 1-s$, we obtain that
\begin{align*}
&\bX \subset C^{0,s - \frac{1}{r}}(\bar I;(\bX_0,\bX_1)_{\theta,1}) \subset C^{0,s - \frac{1}{r}}(\bar I;(\bX_0,\bX_1)_{\theta,2})\\
& = C^{0,s - \frac{1}{r}}(\bar I;\mathbf{W}^{2\theta,4}(\Omega)) \subset C(\bar I;\mathbf{C}^1(\bar\Omega)),
\end{align*}
the embedding $\bX \subset C(\bar I;\mathbf{C}^1(\bar\Omega))$ being compact. We point out that the lower bound $\frac{3}{4} < \theta$  is used to guarantee the continuous embedding $\mathbf{W}^{2\theta,4}(\Omega) \subset  \mathbf{C}^1(\bar\Omega)$.

Now, for every $t \in [0,1]$ and $\bu \in B_{\bar\varepsilon}(\bar\bu)$ we define the linear operators:
\begin{align*}
&L_t:\bX \times \Pi \longrightarrow L^r(I;\bLf),\\
&L_t(\bvarphi,\pi) = -\frac{\partial\bvarphi}{\partial t} - \nu\Delta\bvarphi - t[(\by_\bu \cdot \bna)\bvarphi + (\bna\bvarphi)^T\by_\bu] + \nabla \pi.
\end{align*}
Using the embedding $\bX \subset C(\bar I;\mathbf{C}^1(\bar\Omega))$ and the regularity $\by_\bu \in L^r(I;\bLf)$ established in Lemma \ref{L4.4}, it is obvious that $L_t$ is linear and continuous. Moreover the inyectivity of $L_t$ follows from Theorem \ref{T3.2}. Put $E = \{t \in [0,1] : L_t \text{ is an isomorphism}\}$. The maximal parabolic regularity property of the Stokes system implies that $0 \in E$. Moreover, $E$ is a relatively open set in $[0,1]$. Indeed, if $t_0 \in E$ and $t \in [0,1]$ with $|t - t_0| < \varepsilon$ we have
\begin{align*}
&\|L_t(\bvarphi,\pi) - L_{t_0}(\bvarphi,\pi)\|_{L^r(I;\bLf)} = |t - t_0|\|(\by_\bu \cdot \bna)\bvarphi + (\bna\bvarphi)^T\by_\bu\|_{L^r(I;\bLf)}\\
& \le C_1\varepsilon\|\by_u\|_{L^r(I;\bLf)}\|\bvarphi\|_\bX,
\end{align*}
therefore  $\|L_t - L_{t_0}\|_{\mL(\bX \times \Pi,L^r(I,\bLf)} \le C\varepsilon\|\by_u\|_{L^r(I;\bLf)}$. Since the set of isomorphisms is an open set, we have that $L_t$ is an isomorphism if $\varepsilon$ is small enough. Now, we prove that $E$ is closed. Take a sequence $\{t_k\}_{k = 1}^\infty \subset E$ such that $t_k \to t$. It is enough to prove that $L_t$ is surjective to conclude that $t \in E$. Given an arbitrary element $\bef \in L^r(I;\bLf)$, we introduce the sequence $\{(\bvarphi_k,\pi_k)\}_{k = 1}^\infty \subset \bX \times \Pi$ such that $L_{t_k}(\bvarphi_k,\pi_k) = \bef$. Using the well known estimates for the Stokes system we have
\begin{align*}
&\|(\bvarphi_k,\pi_k)\|_{\bX\times\Pi} \le C_2\|\bef + t_k[(\by_\bu \cdot \bna)\bvarphi_k + (\bna\bvarphi_k)^T\by_\bu]\|_{L^r(I;\bLf)}\\
&\le C_2\left(\|\bef\|_{L^r(I;\bLf)} + 2\|\by_\bu\|_{L^r(I;\bLf)}\|\bvarphi_k\|_{C(\bar I;\mathbf{C}^1(\bar\Omega))}\right).
\end{align*}
Using again Lions's Lemma with the spaces $X \subset C(I;\mathbf{C}^1(\bar\Omega)) \subset \bLqd$ we deduce the existence of a constant  $C_3$ such that
\[
\|(\bvarphi_k,\pi_k)\|_{\bX\times\Pi} \le C_2(\left(\|\bef\|_{L^r(I;\bLf)} + C_3\|\by_\bu\|_{L^r(I;\bLf)}\|\bvarphi_k\|_\bLqd\right) + \frac{1}{2}\|\bvarphi_k\|_\bX,
\]
which proves the boundedness of $\{(\bvarphi_h,\pi_k)\}_{k = 1}^\infty$ in $\bX \times \Pi$. Indeed, the boundedness of $\{\bvarphi_k\}_{k = 1}^\infty$ in $\bLqd$, actually in $\bV^{2,1}(0,T)$, follows from Theorem \ref{T3.2}. Finally, it is straightforward to pass to the limit in $k$ and to conclude that $(\bvarphi_k,\pi_k) \rightharpoonup (\bvarphi,\pi)$ in $\bX \times \Pi$ with $L_t(\bvarphi,\pi) = \bef$. Hence, $L_t$ is also an isomorphism. Since $E$ is nonempty, open, and closed, we conclude that $E= [0,1]$ and, consequently, $\bvarphi_\bu \in \bX$. The estimate \eqref{E4.20} follows from the above estimates.
\end{proof}

\begin{lemma}
Assume that \eqref{E4.7} is fulfilled and let $\bar\varepsilon$ be as introduced in Lemma \ref{L4.4}. Then, for every $\bu \in B_{\bar\varepsilon}(\bar\bu)$ the inequality
\begin{equation}
\|\bvarphi_\bu - \bar\bvarphi\|_{C(\bar I;\mathbf{C}^1(\bar\Omega))} \le M_3\Big(1 + 2M_3\|\by_\bu - \by_d\|_{L^r(I;\bLf)}\Big)\|\by_\bu - \bar\by\|_{L^r(I;\bLf)}
\label{E4.21}
\end{equation}
holds with $M_3$ given by Lemma \ref{L4.5}.
\label{L4.6}
\end{lemma}

\begin{proof}
Taking $(\be,\pi) = (\bvarphi_\bu - \bar\bvarphi,\pi_\bu - \bar\pi)$ and subtracting the corresponding equations we get
\[
\displaystyle-\frac{\partial\be}{\partial t} - \nu\Delta\be - (\bar\by \cdot \bna)\be - (\bna\be)^T\bar\by + \nabla \pi = \by_\bu - \bar\by + [(\by_\bu - \bar\by)\cdot\nabla]\bvarphi_\bu + (\nabla\bvarphi_\bu)^T(\by_\bu - \bar\by) \ \text{ in } Q.
\]
Then, applying Lemma \ref{L4.5} we get
\begin{align*}
&\|\bvarphi_\bu - \bar\bvarphi\|_{C(\bar I;\mathbf{C}^1(\bar\Omega))} \le M_3\|\by_\bu - \bar\by + [(\by_\bu - \bar\by)\cdot\nabla]\bvarphi_\bu + (\nabla\bvarphi_\bu)^T(\by_\bu - \bar\by)\|_{L^r(I;\bLf)}\\
&\le M_3\Big(1 + 2\|\bvarphi_\bu\|_{C(\bar I;\mathbf{C}^1(\bar\Omega))}\Big)\|\by_\bu - \bar\by\|_{L^r(I;\bLf)}\\
& \le M_3\Big(1 + 2M_3\|\by_\bu - \by_d\|_{L^r(I;\bLf)}\Big)\|\by_\bu - \bar\by\|_{L^r(I;\bLf)}.
\end{align*}
\end{proof}

\begin{lemma}
Assume that \eqref{E4.7} holds. Then, for every $\rho > 0$ there exists $\varepsilon_\rho > 0$ such that
\begin{equation}
|[J''(\bu) - J''(\bar\bu)](\bu - \bar\bu)^2| \le \rho\|\bz_{\bu - \bar\bu}\|^2_\bLqd\quad \forall \bu \in \uad \cap \bar B_{\varepsilon_\rho}(\bar\bu),
\label{E4.22}
\end{equation}
where $\bar B_{\varepsilon_\rho}(\bar\bu) = \{\bu \in \uad : \|\bu - \bar\bu\|_{L^q(I;\bWmop)} \le \varepsilon_\rho\}$.
\label{L4.7}
\end{lemma}

\begin{proof}
Let $\bar\varepsilon$ be as defined in Lemma \ref{L4.4} and take $\bu \in \uad \cap B_{\bar\varepsilon}(\bar\bu)$. Let us set $\bv = \bu - \bar\bu$, $\bz_{\bu,\bv} = G'(\bu)\bv$, and $\bz_\bv = G'(\bar\bu)\bv$. According to \eqref{E3.2} we have
\begin{align*}
&|[J''(\bu) - J''(\bar\bu)]\bv^2|\\
& = \left|\int_Q[|\bz_{\bu,\bv}|^2 - 2(\bz_{\bu,\bv}\cdot\bna)\bz_{\bu,\bv}\bvarphi_\bu]\, dx\, dt - \int_Q[|\bz_\bv|^2 - 2(\bz_\bv\cdot\bna)\bz_\bv\bar\bvarphi]\, dx\, dt\right|\\
&\le\int_Q|\bz_{\bu,\bv} + \bz_\bv|\, |\bz_{\bu,\bv} - \bz_\bv|\, dx\, dt +2\left|\int_Q[(\bz_{\bu,\bv} - \bz_\bv)\cdot \bna]\bvarphi_\bu\bz_{\bu,\bv}\, dx\, dt\right|\\
&+ 2\left|\int_Q(\bz_\bv\cdot\bna )(\bvarphi_\bu - \bar\bvarphi)\bz_{\bu,\bv}\, dx\, dt\right| + 2\left|\int_Q(\bz_\bv\cdot\bna )\bar\bvarphi(\bz_{\bu,\bv} - \bz_\bv)\, dx\, dt\right|.
\end{align*}
We estimate the last four integrals. For the first one we use Lemma \ref{L4.2} as follows
\begin{align}
&\int_Q|\bz_{\bu,\bv} + \bz_\bv|\, |\bz_{\bu,\bv} - \bz_\bv|\, dx\, dt \le \|\bz_{\bu,\bv} + \bz_\bv\|_\bLqd\|\bz_{\bu,\bv} - \bz_\bv\|_\bLqd\nonumber\\
&\le 2M_2M_1\|\bu - \bar\bu\|_{L^q(I;\bWmop)}\|\bz_\bv\|^2_\bLqd.
\label{E4.23}
\end{align}

For the second integral we use Lemmas \ref{L4.2} and \ref{L4.5} to get
\begin{align}
&\left|\int_Q[(\bz_{\bu,\bv} - \bz_\bv)\cdot \bna]\bvarphi_\bu\bz_{\bu,\bv}\, dx\, dt\right|\nonumber\\
& \le \|\bvarphi_\bu\|_{C(I;\mathbf{C}^1(\bar\Omega))}\|\bz_{\bu,\bv} - \bz_\bv\|_\bLqd\|\bz_\bv\|_\bLqd\nonumber\\
&\le M_1M_3\|\bu - \bar\bu\|_{L^q(I;\bWmop)}\|\by_\bu - \by_d\|_{L^r(I;\bLf)}\|\bz_\bv\|^2_\bLqd.
\label{E4.24}
\end{align}

The third integral is estimated with Lemmas \ref{L4.2} and \ref{L4.6} as follows
\begin{align}
&\left|\int_Q(\bz_\bv\cdot\bna )(\bvarphi_\bu - \bar\bvarphi)\bz_{\bu,\bv}\, dx\, dt\right| \le \|\bvarphi_\bu - \bar\bvarphi\|_{C(\bar I;\mathbf{C}^1(\bar\Omega))}M_2\|\bz_\bv\|^2_\bLqd\nonumber\\
&\le M_3\Big(1 + 2M_3\|\by_\bu - \bar\by\|_{L^r(I;\bLf)}\Big)\|\by_\bu - \bar\by\|_{L^r(I;\bLf)}M_2\|\bz_\bv\|^2_\bLqd.
\label{E4.25}
\end{align}

The estimate \eqref{E4.24} is also valid for the fourth integral just changing $\by_\bu$ by $\bar\by$. Finally, the existence of $\varepsilon_\rho$ such that \eqref{E4.22} holds is an immediate consequence of the above estimates and Lemma \ref{L4.4}.
\end{proof}

{\em Proof of Theorem \ref{T4.2}.} Using that $G_0'(\bar u):L^q(I;\bWmop) \longrightarrow \mY$ is a linear continuous operator we get
\begin{equation}
\|\bz_{\bu - \bar\bu}\|_\bLqd \le C_\Omega\|\bz_{\bu - \bar\bu}\|_\mY \le C_\Omega\|G_0'(\bar\bu)\|\|\bu - \bar\bu\|_{L^q(I;\bWmop)}.
\label{E4.26}
\end{equation}
From Lemmas \ref{L4.4} and \ref{L4.5}, \eqref{E4.12}, and \eqref{E4.26} we deduce the existence of a constant $M$ such that for every $\bu \in \uad \cap \bar B_{\frac{\bar\varepsilon}{2}}(\bar\bu)$ we have
\begin{align}
&|J''(\bu)(\bu - \bar\bu)^2| \le \Big(1+ 2\|\bvarphi_\bu\|_{C(\bar I;\mathbf{C}^1(\bar\Omega))}\Big)\|\bz_{\bu,\bu - \bar\bu}\|^2_\bLqd\nonumber\\
&\le M\|\bu - \bar\bu\|_{L^q(I;\bWmop)}\|\bz_{\bu - \bar\bu}\|_\bLqd\ \ \forall \bu \in \uad.\label{E4.27}
\end{align}
From Lemma \ref{L4.7} we obtain the existence of $\varepsilon_\delta > 0$ such that
\begin{equation}
|[J''(\bu) - J''(\bar\bu)](\bu - \bar\bu)^2| \le \frac{\delta}{2}\|\bz_{\bu - \bar\bu}\|^2_\bLqd\quad \forall \bu \in \uad \cap \bar B_{\varepsilon_\delta}(\bar\bu),
\label{E4.28}
\end{equation}
where $\delta$ is given in \eqref{E4.8}. We take
\[
\varepsilon = \min\Big\{\frac{\bar\varepsilon}{2},\varepsilon_0,\varepsilon_\delta,\frac{\tau}{2M},\frac{1}{C_\Omega\|G_0'(\bar\bu)\|}\Big\}\ \text{ and } \ \kappa = \min\Big\{\frac{\tau}{4},\frac{\delta}{8}\Big\},
\]
where $\varepsilon_0$ is given in Lemma \ref{L4.3}. Now, we prove the inequality \eqref{E4.9}. To this end, we take $\bu \in \bar B_\varepsilon(\bar\bu) \cap \uad$ and distinguish two cases.

{\em Case I: $\bu - \bar\bu \not\in C^\tau_{\bar\bu}$.} At first we note that if $\|\bar u_i(t)\|_\mo = \gamma$, taking into account that $\bar\bu + \rho(\bu - \bar\bu) \in \uad$ for every $\rho \in (0,1)$, we have
\[
j'(\bar u_i(t);u_i(t) - \bar u_i(t)) = \lim_{\rho \searrow 0}\frac{j(\bar u_i(t) + \rho(u_i(t) - \bar u_i(t))) - \gamma}{\rho} \le 0 \text{ \rm for } i=1, 2.
\]
Therefore, if $\bu - \bar\bu \not\in C^\tau_{\bar\bu}$, then one (or both) of the two conditions holds
\begin{align}
&\text{\rm I)}\ \frac{\partial\mL}{\partial\bu}(\bar\bu,\bar\bpsi)(\bu - \bar\bu) > \tau\|\bz_{\bu - \bar\bu}\|_\bLqd,\label{E4.29}\\
&\text{\rm II)} -\tau\|\bz_{\bu-\bar\bu}\|_\bLqd > \sum_{i = 1}^2\int_0^T\|\bar\varphi_i(t)\|_{C_0(\omega)} j'(\bar u_i(t);v_i(t))\, dt.\label{E4.30}
\end{align}

If \eqref{E4.29} holds, then performing a Taylor expansion of $J$ around $\bar\bu$, using the convexity of $j$, \eqref{E3.14} and $\|\bu(t)\|_\bM \le \gamma$, \eqref{E4.26}, \eqref{E4.27}, and taking into account the definitions of $\varepsilon$ and $\kappa$, we get for some $\theta \in [0,1]$
\begin{align*}
&J(\bu) - J(\bar\bu) \ge \mL(\bu,\bar\bpsi) - \mL(\bar\bu,\bar\bpsi) \ge \frac{\partial\mL}{\partial\bu}(\bar\bu,\bar\bpsi)(\bu - \bar\bu) + \frac{1}{2}J''(\bar\bu + \theta(\bu - \bar\bu))(\bu - \bar\bu)^2\\
&\ge\tau\|\bz_{\bu - \bar\bu}\|_\bLqd -\frac{\tau}{2}\|\bz_{\bu - \bar\bu}\|_\bLqd = \frac{\tau}{2}\|\bz_{\bu - \bar\bu}\|_\bLqd\ge \frac{\tau}{2}\|\bz_{\bu - \bar\bu}\|_\bLqd^2\\
& \ge \frac{\tau}{8}\|\by_\bu - \bar\by\|^2_\bLqd\ge \frac{\kappa}{2}\|\by_\bu - \bar\by\|^2_\bLqd.
\end{align*}

If \eqref{E4.30} holds, then we obtain $J'(\bar\bu)(\bu - \bar\bu) > \tau\|\bz_{\bu- \bar\bu}\|_\bLqd$  due to \eqref{E3.20}. Then, this inequality, \eqref{E4.26}, and \eqref{E4.27} yield
\begin{align*}
&J(\bu) - J(\bar\bu) = J'(\bar\bu)(\bu - \bar\bu) + \frac{1}{2}J''(\bar\bu + \theta(\bu - \bar\bu))(\bu - \bar\bu)^2\\
&> \tau\|\bz_{\bu - \bar\bu}\|_\bLqd - \frac{\tau}{2}\|\bz_{\bu - \bar\bu}\|_\bLqd \ge \frac{\kappa}{2}\|\by_\bu - \bar\by\|^2_\bLqd.
\end{align*}

{\em Case II: $\bu - \bar\bu \in C^\tau_{\bar\bu}$.} We use $J'(\bar\bu)(\bu - \bar\bu) \ge 0$, \eqref{E4.8}, and \eqref{E4.28} to infer
\begin{align*}
&J(\bu) - J(\bar\bu) = J'(\bar\bu)(\bu - \bar\bu) + \frac{1}{2}J''(\bar\bu + \theta(\bu - \bar\bu))(\bu - \bar\bu)^2\\
&\ge \frac{1}{2}J''(\bar\bu)(\bu - \bar\bu)^2 + \frac{1}{2}[J''(\bar\bu + \theta(\bu - \bar\bu)) - J''(\bar\bu)](\bu - \bar\bu)^2\\
& \ge \frac{\delta}{2}\|\bz_{\bu - \bar\bu}\|^2_\bLqd - \frac{\delta}{4}\|\bz_{\bu - \bar\bu}\|^2_\bLqd = \frac{\delta}{4}\|\bz_{\bu - \bar\bu}\|^2_\bLqd\\
& \ge  \frac{\delta}{16}\|\by_\bu - \bar\by\|^2_\bLqd \ge \frac{\kappa}{2}\|\by_\bu - \bar\by\|^2_\bLqd,
\end{align*}
which concludes the proof.


\begin{thebibliography}{10}

\bibitem{Abergel-Temam90}
F.~Abergel and R.~Temam.
\newblock On some control problems in fluid mechanics.
\newblock {\em Theoret. Comput. Fluid Dynamics}, 1:303--325, 1990.

\bibitem{Amann1995}
H.~Amann.
\newblock {\em Linear and quasilinear parabolic problems. Vol. 1}.
\newblock {Birkh\"auser}, Boston, 1995.

\bibitem{Amann2001}
H.~Amann.
\newblock Linear parabolic problems involving measures.
\newblock {\em Rev. R. Acad. Cien. Serie A. Mat.}, 95(1):85--119, 2001.

\bibitem{BTZ00}
T.~Bewley, R.~Temam, and M.~Ziane.
\newblock Existence and uniqueness of optimal control to the {N}avier-{S}tokes
  equations.
\newblock {\em Comptes Rendus de l'Acad\'{e}mie des Sciences. S\'{e}rie I.
  Math\'{e}matique}, 330(11):1007--1011, 2000.

\bibitem{Boyer-Fabrie2013}
F.~Boyer and P.~Fabrie.
\newblock {\em Mathematical Tools for the Study of the Incompressible
  {N}avier-{S}tokes Equations and Related Models}.
\newblock Springer, New York, 2013.

\bibitem{Casas2017}
E.~Casas.
\newblock A review on sparse solutions in optimal control of partial
  differential equations.
\newblock {\em SEMA Journal}, 2017.

\bibitem{CC16}
E.~Casas and K.~Chrysafinos.
\newblock A review of numerical analysis for the discretization of the velocity
  tracking problem.
\newblock In {\em Trends in differential equations and applications}, volume~8
  of {\em SEMA SIMAI Springer Ser.}, pages 51--71. Springer, [Cham], 2016.

\bibitem{CCK2012}
E.~Casas, C.~Clason, and K.~Kunisch.
\newblock Approximation of elliptic control problems in measure spaces with
  sparse solutions.
\newblock {\em SIAM J. Control Optim.}, 50(4):1735--1752, 2012.

\bibitem{CCK2013}
E.~Casas, C.~Clason, and K.~Kunisch.
\newblock Parabolic control problems in measure spaces with sparse solutions.
\newblock {\em SIAM J. Control Optim.}, 51(1):28--63, 2013.

\bibitem{CHW2017}
E.~Casas, R.~Herzog, and G.~Wachsmuth.
\newblock Analysis of spatio-temporally sparse optimal control problems of
  semilinear parabolic equations.
\newblock {\em ESAIM Control Optim. Calc. Var.}, 23:263--295, 2017.

\bibitem{Casas-Kunisch2014}
E.~Casas and K.~Kunisch.
\newblock Optimal control of semilinear elliptic equations in measure spaces.
\newblock {\em SIAM J. Control Optim.}, 52(1):339--364, 2014.

\bibitem{Casas-Kunisch2019}
E.~Casas and K.~Kunisch.
\newblock Optimal control of the 2d stationary navier-stokes equations with
  measure valued controls.
\newblock {\em SIAM J. Control Optim.}, 57(2):1328--1354, 2019.

\bibitem{Casas-Kunisch2019B}
E.~Casas and K.~Kunisch.
\newblock Using sparse control methods to identify sources in linear
  diffusion-convection equations.
\newblock {\em Inverse Problems}, 35(11):114002, 17, 2019.

\bibitem{Casas-Kunisch2020}
E.~Casas and K.~Kunisch.
\newblock Well-posedness of evolutionary navier-stokes equations with forces of
  low regularity on two-dimensional domains.
\newblock {\em {arXiv:2004.10456 [math.AP]}}, 2020.

\bibitem{DG08}
J.~C. De~Los~Reyes and R.~Griesse.
\newblock State-constrained optimal control of the three-dimensional stationary
  {N}avier-{S}tokes equations.
\newblock {\em Journal of Mathematical Analysis and Applications},
  343(1):257--272, 2008.

\bibitem{DI94}
M.~Desai and K.~Ito.
\newblock Optimal controls of {N}avier-{S}tokes equations.
\newblock {\em SIAM Journal on Control and Optimization}, 32(5):1428--1446,
  1994.

\bibitem{Edwards1965}
R.E. Edwards.
\newblock {\em Functional Analysis}.
\newblock Holt, Rinehart and Winston, New York, 1965.

\bibitem{Ekeland-Temam76}
I.~Ekeland and R.~Temam.
\newblock {\em Convex Analysis and Variational Problems}.
\newblock North-Holland-Elsevier, New York, 1976.

\bibitem{G12}
M.~D. Gunzburger.
\newblock {\em Perspectives in flow control and optimization}, volume~5 of {\em
  Advances in Design and Control}.
\newblock Society for Industrial and Applied Mathematics (SIAM), Philadelphia,
  PA, 2003.

\bibitem{HSW2015}
R.~Herzog, G.~Stadler, and G.~Wachsmuth.
\newblock Erratum: {D}irectional sparsity in optimal control of partial
  differential equations.
\newblock {\em SIAM J. Control Optim.}, 53(4):2722--2723, 2015.

\bibitem{HK01}
M.~Hinze and K.~Kunisch.
\newblock Second order methods for optimal control of time-dependent fluid
  flow.
\newblock {\em SIAM J. Control Optim.}, 40(3):925--946, 2001.

\bibitem{KT2016}
K.~Kunisch, Ph. Trautmann, and B.~Vexler.
\newblock Optimal control of the undamped linear wave equation with measure
  valued controls.
\newblock {\em SIAM J. Control Optim.}, 54(3):1212--1244, 2016.

\bibitem{Lions-Magenes68}
J.L. Lions and E.~Magenes.
\newblock {\em Probl\`{e}mes aux Limites non Homog\`{e}nes. Volume 1}.
\newblock Dunod, Paris, 1968.

\bibitem{Rudin70}
W.~Rudin.
\newblock {\em Real and Complex Analysis}.
\newblock McGraw-Hill Book Co., London, 1970.

\bibitem{Serre1983}
D.~Serre.
\newblock {\'E}quations de {N}avier-{S}tokes stationnaries avec données peu
  reguli\'eres.
\newblock {\em Ann. Sci. Norm. Sup. Pisa}, 10:543--559, 1983.

\bibitem{Sohr2001}
H.~Sohr.
\newblock {\em The {N}avier-{S}tokes equations}.
\newblock Birkh\"{a}user Advanced Texts: Basler Lehrb\"{u}cher. [Birkh\"{a}user
  Advanced Texts: Basel Textbooks]. Birkh\"{a}user Verlag, Basel, 2001.
\newblock An elementary functional analytic approach.

\bibitem{Temam79}
R.~Temam.
\newblock {\em {N}avier-{S}tokes Equations}.
\newblock North-Holland, Amsterdam, 1979.

\bibitem{Triebel1978}
H.~Triebel.
\newblock {\em Interpolation Theory, Function Spaces, Differential Operators}.
\newblock Notrh-Holland, Berlin, 1978.

\bibitem{TW06}
F.~Tr\"{o}ltzsch and D.~Wachsmuth.
\newblock Second-order sufficient optimality conditions for the optimal control
  of {N}avier-{S}tokes equations.
\newblock {\em ESAIM. Control, Optimisation and Calculus of Variations},
  12(1):93--119, 2006.

\end{thebibliography}
\end{document}